\documentclass[preprint,1p,times]{elsarticle}
\usepackage{geometry}

\usepackage{graphicx} 
\usepackage{chemformula} 
\usepackage[T1]{fontenc} 

\usepackage{natbib}
\usepackage{amsmath,amsthm,amssymb}
\usepackage{float}
\usepackage{graphicx}
\usepackage{caption}
\usepackage{subfigure}
\usepackage{tabularx}
\usepackage{threeparttable}
\usepackage{dcolumn}
\usepackage{bm}
\usepackage{array}
\usepackage{hyperref}
\usepackage{algpseudocode}
\usepackage{algorithmicx,enumerate,algorithm}
\usepackage{booktabs}
\usepackage{threeparttable}
\usepackage{makecell}

\newcommand{\cO}{\mathcal{O}}

\newtheorem{theorem}{Theorem}

\theoremstyle{definition}

\theoremstyle{remark}
\newtheorem{remark}[theorem]{Remark}
\newtheorem{proposition}{Proposition}

\newtheorem{corollary}{Corollary}

\newcommand{\tcb}[1]{\textcolor{black}{#1}}
\newcommand{\tcr}[1]{\textcolor{black}{#1}}

\allowdisplaybreaks[4]

\journal{Journal of Computational Physics}

\begin{document}

\begin{frontmatter}

\title{
Fast spectral separation method for kinetic equation with anisotropic non-stationary collision operator retaining micro-model fidelity \\
} 
\author[1]{Yue Zhao}
\ead{zhaoyu14@msu.edu}
\address[1]{{Department of Computational Mathematics, Science \& Engineering, Michigan State University}, 428 S Shaw Ln, East Lansing 48824, MI, USA}
\author[1,2]{Huan Lei\corref{cor2}}
\cortext[cor2]{Corresponding author.}
\ead{leihuan@msu.edu}
\address[2]{{Department of Statistics \& Probability, Michigan State University}, 619 Red Cedar Road Wells Hall, East Lansing 48824, MI, USA}

\date{\today}


\begin{abstract}
In this paper, we present a generalized, data-driven collisional operator for one-component plasmas, learned from molecular dynamics simulations, to extend the collisional kinetic model beyond the weakly coupled regime.
The proposed operator features an anisotropic, non-stationary collision kernel that accounts for particle correlations typically neglected in classical Landau formulations.
To enable efficient numerical evaluation, we develop a fast spectral separation method that represents the kernel as a low-rank tensor product of univariate basis functions. 
This formulation admits an $\cO(N \log N)$ algorithm via fast Fourier transforms and preserves key physical properties, including discrete conservation laws and the H-theorem, through a structure-preserving central difference discretization. 
Numerical experiments demonstrate that the proposed model accurately captures plasma dynamics in the moderately coupled regime beyond the standard Landau model while maintaining high computational efficiency and structure-preserving properties.

\end{abstract}



\begin{keyword}
Kinetic equation, collision operator, data-driven modeling, non-stationary, spectral separation.
\end{keyword}

\end{frontmatter}


\section{Introduction}

Collisional kinetic theory provides an effective computational framework for modeling dynamical processes that bridge microscopic particle interactions with macroscopic transport phenomena in fields such as plasma physics, gas dynamics, and astrophysics. Besides the Vlasov term, the governing kinetic equation includes a collision operator that captures the unresolved particle interactions beyond the mean-field approximation. A widely adopted choice is the Landau collision operator \cite{landau1937kinetic}, which models binary particle collisions under long-range (e.g., Coulomb) interactions through a three-dimensional integro-differential equation.  It can be formally derived as a limiting case of more detailed models, such as the five-dimensional Boltzmann collision operator \cite{boltzmann1872weitere, wild1951boltzmann} in the grazing collision limit, or the seven-dimensional Balescu–Lenard operator \cite{lenard1960bogoliubov, balescu1960irreversible} under the assumption of an isotropic dielectric response. The Landau operator provides an accurate description of plasma kinetics in the weakly coupled regime and strictly preserves conservation laws as well as the H-theorem. 

The numerical solution of the Landau equation is challenging due to its intrinsic physical structure, high-dimensional integral formulation, and nonlinear stiffness. Substantial effort has been devoted to developing efficient and accurate numerical methods for this equation. Positivity-preserving schemes have been developed for both the linearized Fokker-Planck equation \cite{chang1970practical} and the full nonlinear form \cite{larsen1985discretization}. The entropy-based schemes proposed in \cite{buet1999numerical,degond1994entropy} are designed to rigorously preserve key physical quantities while ensuring non-negative entropy dissipation. 
However, direct implementation of these schemes can be computationally expensive due to the high-dimensional integral form of the collision operator. To address this, several fast evaluation techniques have been developed. These include the multigrid algorithms \cite{buet1997fast}, multipole expansions \cite{lemou1998multipole}, the Fourier spectral method \cite{pareschi2000fast,filbet2002numerical,zhang2017conservative} that exploit the convolution structure of the Landau kernel, the \tcr{particle-based method \cite{Carrillo_JCP_2020, carrillo2021random, bailo2024collisional}} with regularized entropy functional, as well as the approaches \cite{chacon2000implicit,taitano2016adaptive} utilizing the Rosenbluth formulation \cite{Rosenbluth_PR_1957} via Poisson solvers. 
Recent works developed \tcr{Hermite spectral based methods \cite{wang2019approximation,li2020approximation,li2021hermite}} that efficiently evaluate the spectral coefficients for the quadratic collision terms. To handle the stiffness arising from the nonlinear structure of the equation, implicit schemes \cite{epperlein1994implicit,chacon2000implicit,lemou2005implicit,taitano2016adaptive} and the asymptotic-preserving approaches \cite{filbet2010class,jin2011class} have been proposed to accelerate the convergence and improve stability in regimes with disparate temporal scales.

Despite the significant theoretical and computational advances, the applicability of the Landau collisional operator remains limited to the weakly coupled regime, where the ratio between the kinetic and potential energy is much larger than one. Due to the assumption of small-angle scattering and uncorrelated particle interactions, the operator exhibits limitations in modeling the kinetic processes in regimes where particle correlations become non-negligible. This limitation poses severe challenges for realistic applications such as inertial confinement fusion \cite{Rinderknecht_kinetic_plasma_reivew_2018, Indirect_drive_ICF_PRL_1_2024} and photon scattering \cite{Dharma_Perrot_PRE_1998} that span a broader range of physical regimes. To overcome these challenges, we proposed a generalized data-driven collision operator \cite{zhao2025data} directly learned from the microscale molecular dynamics (MD) simulations. In contrast to the Landau form with an isotropic and stationary collision kernel, the proposed collision operator takes a generalized metriplectic structure \cite{morrison1986paradigm} with an anisotropic kernel. In particular, the kernel takes a non-stationary form that depends not only on the relative velocity but also on the average velocity of the colliding pair. This generalized formulation strictly preserves the frame-indifference and various physical constraints. Moreover, it enables us to faithfully account for the heterogeneous energy transfer in the plane perpendicular to the relative velocity. This heterogeneity effect arises from the collective interactions between the pair of collision particles and the surrounding environment, which have been broadly overlooked in most existing forms but prove crucial for capturing the plasma kinetics beyond the weakly coupled regime \cite{zhao2025data}, where the Landau model shows limitations.  

While the proposed generalized operator significantly extends the modeling capability of collisional kinetics, one remaining challenge is the efficient evaluation of the new collision operator. Due to the anisotropic and non-stationary form of the kernel, the fast approximation based on the convolution structure or the Rosenbluth formulation can not be directly applied; direct evaluation requires $\cO(N^2)$ computational complexity, where $N$ denotes the number of velocity grid points. In this study, we aim to fill the gap by developing an efficient approach based on the fast spectral separation (SS) method. The main idea is to seek a spectral representation of the anisotropic collision kernel as the low-rank tensor product of a set of univariate basis functions. Each basis function depends on a scalar argument, such as the magnitude of the velocity of one of the colliding particles or their relative velocity.
This tensor representation strictly preserves the symmetry constraints and enables spectral separation of the individual and relative velocity dependencies.  By embedding the kernel representation into the bilinear metriplectic form of the collision operator, we demonstrate that each integral term inherits a convolution structure, and therefore achieves $\cO(N\log N)$ computational complexity via the fast Fourier transform (FFT) method. The tensor representation in the form of the univariate basis can be directly learned from the MD data, allowing the collision model to faithfully capture the effects of micro-scale particle correlations beyond the Landau model. Meanwhile, the constructed model strictly satisfies the conservation laws and the H-theorem, which enables the development of structure-preserving numerical schemes for the present collisional kinetic equation following the essential ideas for the Landau model. Numerical results show that the present model improves the Landau equation in comparison with full MD simulation results, and achieves efficient computational complexity with structure-preserving numerical solution. 

The rest of the manuscript is organized as follows.
In Section \ref{sec:method}, we introduced the generalized collision operator structure with an anisotropic non-stationary kernel, the spectral separation representation and fast approximation method, as well as the structure-preserving discretization scheme. 
Section \ref{sec:numerical} presents the numerical results that demonstrate the accuracy of the collision operator, the computational efficiency of the fast approximation method, and the convergence of the collisional kinetic model. Summary is given in Section \ref{sec:summary}.






\section{Methods}\label{sec:method}

The collisional Vlasov-Poisson equation takes a general form as
\begin{equation}\label{eq:VP}
	\dfrac{\partial f}{\partial t}(\boldsymbol{z},t) +\boldsymbol{v} \cdot \dfrac{\partial f}{\partial \boldsymbol{x}} - \dfrac{\partial \phi}{\partial \boldsymbol{x}}(\boldsymbol{x}; f) \cdot \dfrac{\partial f}{\partial \boldsymbol{v}} =\left.\dfrac{\partial f}{\partial t}\right|_{c},
\end{equation}
where $f:\mathbb{R}^6\times \mathbb{R}_{+} \to \mathbb{R}_{+}$ is the single-particle probability density function in the phase space and $\boldsymbol{z}=(\boldsymbol{x},\boldsymbol{v})$ represents a phase space point with position $\boldsymbol{x} \in \mathbb{R}^3$ and $\boldsymbol{v} \in \mathbb{R}^3$.  $\phi(\boldsymbol{x} ;f)=\int V(\boldsymbol{x},\boldsymbol{x}')f(\boldsymbol{z}')\mathrm{d}\boldsymbol{z}'$ is the mean field potential energy with pair potential $V(\boldsymbol{x}, \boldsymbol{x}')$.
The collision term $\left.\dfrac{\partial f}{\partial t}\right|_{c}$ represents the unresolved particle interactions beyond the mean field approximation. In this work, we focus on the spatially homogeneous scenario and abusively use $f(\boldsymbol{v}, t)$ to denote the velocity distribution function for the remainder of this work; the kinetic equation reduces to
\begin{equation}\label{eq:collision}
	\dfrac{\partial f}{\partial t}(\boldsymbol{v},t) = \left.\dfrac{\partial f}{\partial t}\right|_{c}.
\end{equation}
In principle, the collision term can be derived from the Liouville equations with the BBGKY hierarchy \cite{kirkwood1946statistical, kirkwood1947statistical, bogoliubov1947kinetic, bhatnagar1954model}. However, the derived form is generally unclosed and depends on the higher-order correlation functions. In practice, various empirical closed forms $C[f]$, such as the Landau, BGK, and Boltzmann collision operators, are often introduced based on different heuristic approximations. In this work, we choose a different pathway and aim to construct a generalized collision operator directly from the micro-scale MD data.

\subsection{Generalized anisotropic and non-stationary collision operator}
\label{sec:generalized_kernel}
To capture the anisotropic nature of energy transfer in the velocity space, we represent the generalized collision operator with the metriplectic structure \cite{morrison1986paradigm} in the form of 
\begin{equation}\label{eq:collision_symplectic}
    \dfrac{\partial f}{\partial t}(\boldsymbol{v},t) = \nabla \cdot \int \boldsymbol{\omega}(\boldsymbol{v}, \boldsymbol{v}') \left[ f(\boldsymbol{v}') \nabla f(\boldsymbol{v}) - f(\boldsymbol{v}) \nabla' f(\boldsymbol{v}') \right] \mathrm{d}\boldsymbol{v}',
\end{equation}
\tcb{which can be constructed by the dissipative bracket $(f, \mathcal{S})_{+}$ with the entropy $\mathcal{S} = -\int f\ln f \mathrm{d}\boldsymbol{v}$ and 
$(g, h)_{+} := \frac{1}{2}
\int \int [\nabla \frac{\delta g}{\delta f} - \nabla' \frac{\delta g}{\delta f}] f f' \bm\omega(\bm v, \bm v') [\nabla \frac{\delta h}{\delta f} - \nabla' \frac{\delta h}{\delta f}] \mathrm{d}\boldsymbol{v} \mathrm{d}\boldsymbol{v}'
$.} In particular, $\boldsymbol{\omega}: \mathbb{R}^3\times\mathbb{R}^3 \to \mathbb{S}^{3}_{+}$ is a collision kernel representing the unresolved particle interactions. As shown later, $\boldsymbol{\omega}(\boldsymbol{v}, \boldsymbol{v}')$ takes a generalized anisotropic non-stationary structure and differs from the stationary form $\boldsymbol{\omega}(\boldsymbol{u})$ where $\boldsymbol{u} = \boldsymbol{v} - \boldsymbol{v}'$. As a special case, by choosing $\boldsymbol{\omega}(\boldsymbol{v}, \boldsymbol{v}') \propto \vert\boldsymbol{u}\vert^{-1}\boldsymbol{\mathcal{P}}$ and $\boldsymbol{\mathcal{P}} = (\boldsymbol{I} - \boldsymbol{u}\boldsymbol{u}^T/\vert \boldsymbol{u}\vert ^{2})$, the operator recovers the standard Landau model with an isotropic energy transfer  ($\propto \vert\boldsymbol{u}\vert^{-1}$) in the plane perpendicular to $\boldsymbol{u}$.

To construct a generalized form, the collision kernel $\boldsymbol{\omega}(\boldsymbol{v}, \boldsymbol{v}')$ needs to satisfy certain properties such that the collision operator in the form of Eq. \eqref{eq:collision_symplectic} satisfies the frame-indifference and physical constraints. \tcr{Specifically, 
the frame-indifference means that the evolution of the dynamic equations remains invariant under rigid motions (translation and rotation) of the coordinate system and velocity reflections:
\begin{equation*}
    \begin{aligned}
        \left.\dfrac{\partial \tilde{f}}{\partial t}\right|_{c} (\tilde{\boldsymbol{x}},\tilde{\boldsymbol{v}},t) &= \left.\dfrac{\partial f}{\partial t}\right|_{c} (\boldsymbol{x},\boldsymbol{v},t), \quad \tilde{\boldsymbol{x}} = \mathcal{U}\boldsymbol{x} + \boldsymbol{x}_{0}, ~ \tilde{\boldsymbol{v}} = \mathcal{U}\boldsymbol{v}, \\
        \left.\dfrac{\partial f^*}{\partial t}\right|_{c} (\boldsymbol{v}^*,t) &= \left.\dfrac{\partial f}{\partial t}\right|_{c} (\boldsymbol{v},t), \quad\quad \boldsymbol{v}^* = -\boldsymbol{v},
    \end{aligned}
\end{equation*}
where $\boldsymbol{x}_0$ and $\mathcal{U} \in {\rm SO}(3)$ specify the translation and rotation.}

\begin{proposition}
\label{prop:kernel_condition}
Let the collision kernel $\boldsymbol{\omega}(\boldsymbol{v},\boldsymbol{v}')$ be a symmetric positive semi-definite kernel and satisfy the following properties: permutation invariance, rotational symmetry, and orthogonality to the relative velocity $\boldsymbol{u} = \boldsymbol{v} - \boldsymbol{v}'$, i.e., 
\begin{equation}\label{eq:kernel_conditions}
    \begin{aligned}
        \boldsymbol{\omega}(\mathcal{U}\boldsymbol{v},\mathcal{U}\boldsymbol{v}') &= \mathcal{U}\boldsymbol{\omega}(\boldsymbol{v},\boldsymbol{v}')\mathcal{U}^T,~  \mathcal{U} \in {\rm SO}(3)\\
        \boldsymbol{\omega}(\boldsymbol{v},\boldsymbol{v}') &= \boldsymbol{\omega}(\boldsymbol{v}',\boldsymbol{v}) \\
        \boldsymbol{\omega}(\boldsymbol{v},\boldsymbol{v}')(\boldsymbol{v} -\boldsymbol{v}') &= \boldsymbol{0} .
    \end{aligned}
\end{equation}
Then the kinetic model \eqref{eq:collision_symplectic} strictly conserves the mass, momentum, and energy, preserves the frame indifference, and satisfies the H-theorem.     
\end{proposition}

\begin{proof}
    The total mass, momentum and kinetic energy are 
    \begin{equation}\label{eq:MPK}
        m = \int f \mathrm{d} \boldsymbol{v}, ~ \boldsymbol{p} = \int \boldsymbol{v}f \mathrm{d} \boldsymbol{v}, ~ K = \frac{1}{2}\int \boldsymbol{v}^2 f \mathrm{d} \boldsymbol{v}.
    \end{equation}
    For any physical quantity $\phi$, we have
    \begin{equation}\label{eq:conservation}
        \begin{aligned}
            \dfrac{\mathrm{d}}{\mathrm{d} t} \int \phi(\boldsymbol{v}) f \mathrm{d} \boldsymbol{v} &= \int \phi(\boldsymbol{v}) \left(\nabla \cdot \int \boldsymbol{\omega} (\boldsymbol{v},\boldsymbol{v}') \left[ f(\boldsymbol{v}') \nabla f(\boldsymbol{v}) - f(\boldsymbol{v}) \nabla' f(\boldsymbol{v}') \right] \mathrm{d}\boldsymbol{v}'\right) \mathrm{d}\boldsymbol{v}\\
            &= -\dfrac{1}{2}\iint \left(\nabla_{\boldsymbol{v}} \phi(\boldsymbol{v}) - \nabla_{\boldsymbol{v}'} \phi(\boldsymbol{v}')\right) \boldsymbol{\omega}(\boldsymbol{v},\boldsymbol{v}') \left[ f(\boldsymbol{v}') \nabla f(\boldsymbol{v}) - f(\boldsymbol{v}) \nabla' f(\boldsymbol{v}') \right] \mathrm{d}\boldsymbol{v}' \mathrm{d}\boldsymbol{v},
        \end{aligned}
    \end{equation}
    so the quantities $(m,\boldsymbol{p},K)$ are conserved.
    The entropy of the system and its evolution are
    \begin{equation}\label{eq:Hthm}
        \begin{aligned}
            S(f)&=-\int f \log f \mathrm{d}\boldsymbol{v},\\
            \dfrac{\mathrm{d}S}{\mathrm{d}t} &= \dfrac{1}{2}\iint  \boldsymbol{B}^{T} \boldsymbol{\omega}(\boldsymbol{v},\boldsymbol{v}') \boldsymbol{B} f(\boldsymbol{v})f(\boldsymbol{v}') \mathrm{d}\boldsymbol{v}' \mathrm{d}\boldsymbol{v} \geq 0,
        \end{aligned}
    \end{equation}
    where $\boldsymbol{B} = \nabla_{\boldsymbol{v}} \log f(\boldsymbol{v}) - \nabla_{\boldsymbol{v}'} \log f(\boldsymbol{v}')$. Furthermore, it ensures a non-negative solution and admits the Maxwellian distribution as the equilibrium state. In addition, we can show that the collision kernel needs to satisfy the symmetry condition $\boldsymbol{\omega}(\boldsymbol{v},\boldsymbol{v}') = \boldsymbol{\omega}(-\boldsymbol{v},-\boldsymbol{v}')$.
\end{proof}

Based on Proposition \ref{prop:kernel_condition}, we propose a generalized anisotropic and non-stationary collision kernel \tcb{admitting a symmetry-breaking form that not only depends on $\boldsymbol{v} - \boldsymbol{v}'$ but also the collective motion $\boldsymbol{v} + \boldsymbol{v}'$ to capture the broadly overlooked heterogeneous nature of the collisional energy transfer. Meanwhile, it strictly} satisfies Eq. \eqref{eq:kernel_conditions} in the form of 
\begin{equation}\label{eq:CM2}
    \boldsymbol{\omega} (\boldsymbol{v}, \boldsymbol{v}') = \boldsymbol{\mathcal{P}} \left(g_{r}^2 \widetilde{\boldsymbol{r}}\widetilde{\boldsymbol{r}}^T + g_{s}^2 \widetilde{\boldsymbol{s}}\widetilde{\boldsymbol{s}}^T\right) \boldsymbol{\mathcal{P}},
\end{equation}
where $\boldsymbol{r}=\boldsymbol{v}+\boldsymbol{v}' - 2\bar{\boldsymbol{v}}$, $\boldsymbol{s}=\boldsymbol{u}\times\boldsymbol{r}$ and $\boldsymbol{\mathcal{P}}=\boldsymbol{I}-\boldsymbol{u}\boldsymbol{u}^T/|\boldsymbol{u}|^2$. $\bar{\boldsymbol{v}} = m^{-1} \boldsymbol{p}$ is the average velocity. In this work, we assume $\bar{\boldsymbol{v}} \equiv 0$ for simplicity.  Furthermore, we denote $\widetilde{\boldsymbol{u}}=\boldsymbol{u}/|\boldsymbol{u}|$, $\widetilde{\boldsymbol{r}}=\boldsymbol{\mathcal{P}}\boldsymbol{r}/|\boldsymbol{\mathcal{P}}\boldsymbol{r}|$ and $\widetilde{\boldsymbol{s}}=\boldsymbol{s}/|\boldsymbol{s}|$ as mutually orthogonal unit vectors, satisfying $\widetilde{\boldsymbol{u}} \widetilde{\boldsymbol{u}}^{T} + \widetilde{\boldsymbol{r}} \widetilde{\boldsymbol{r}}^{T} + \widetilde{\boldsymbol{s}} \widetilde{\boldsymbol{s}}^{T} = \boldsymbol{I}$.  The functions $g_{r}$ and $g_{s}$ represent the magnitude of the energy transfer along $\widetilde{\boldsymbol{r}}$ and $\widetilde{\boldsymbol{s}}$ directions. They should be rotational invariant and therefore constructed as $g_r\left (\vert \boldsymbol{u}\vert, \vert \boldsymbol{r}\vert, \vert \boldsymbol{s}\vert \right)$ and $g_s\left (\vert \boldsymbol{u}\vert, \vert \boldsymbol{r}\vert, \vert \boldsymbol{s}\vert \right)$, \tcb{preserving the frame-indifference condition \eqref{eq:kernel_conditions}}.

We emphasize that the anisotropic form of the new collision kernel does not violate the frame-indifference constraint. In particular, the solution of the new kinetic model strictly preserves the isotropy of the initial condition. 
\begin{corollary}
\label{cor:radial_symmetry}
If $f(\boldsymbol{v}, 0)$ is isotropic, the solution $f(\boldsymbol{v}, t)$ of model  \eqref{eq:collision_symplectic} \eqref{eq:CM2} remains isotropic for $t > 0$.   
\end{corollary}
\begin{proof}
Let us consider two points $\boldsymbol{v}_1 = \mathcal{U} \boldsymbol{v}_2$, where $\mathcal{U} \mathcal{U}^T = \boldsymbol{I}$. It is easy to show 
\begin{equation*}
\begin{split}
C[f](\boldsymbol{v}_1,0) &= \nabla_{\boldsymbol{v}_1} \cdot \int \boldsymbol\omega(\boldsymbol{v}_1, \boldsymbol{v}_1') \left[ f(\boldsymbol{v}_1') \nabla f(\boldsymbol{v}_1) - f(\boldsymbol{v}_1) \nabla' f(\boldsymbol{v}_1') \right] \mathrm{d}\boldsymbol{v}_1' \\
&= (\mathcal{U}  \nabla_{\boldsymbol{v}_2}) \cdot \int \boldsymbol\omega(\mathcal{U} \boldsymbol{v}_2, \mathcal{U}\boldsymbol{v}_2') \left[ f(\boldsymbol{v}_2') \mathcal{U} \nabla f(\boldsymbol{v}_2) - f(\boldsymbol{v}_2) \mathcal{U} \nabla' f(\boldsymbol{v}_2') \right] \mathrm{d}\boldsymbol{v}_2' \\
&= (\mathcal{U}  \nabla_{\boldsymbol{v}_2}) \cdot \int \mathcal{U} \boldsymbol\omega( \boldsymbol{v}_2, \boldsymbol{v}_2') \mathcal{U}^T \left[ f(\boldsymbol{v}_2') \mathcal{U} \nabla f(\boldsymbol{v}_2) - f(\boldsymbol{v}_2) \mathcal{U} \nabla' f(\boldsymbol{v}_2') \right] \mathrm{d}\boldsymbol{v}_2' \\
&\equiv C[f](\boldsymbol{v}_2, 0),
\end{split}    
\end{equation*}
\end{proof}

As shown in Ref. \cite{zhao2025data}, the generalized collision operator \eqref{eq:collision_symplectic} \eqref{eq:CM2} can capture the heterogeneous energy transfer arising from the collective interactions between pairs of collision particles and the environment and improve the Landau collision operator in a significantly broader range of physical conditions. However, the kernel $\omega(\boldsymbol{v}, \boldsymbol{v}')$ in Eq. \eqref{eq:CM2} is non-stationary and therefore does not preserve the convolution structure. The direct numerical evaluation of the collision operator takes $O(N^2)$ computational complexity per timestep.

\subsection{Spectral separation representation}
To address the aforementioned challenge, we seek a new form of the kernel based on the spectral separation (SS) representation. Since $\widetilde{\boldsymbol{u}} \widetilde{\boldsymbol{u}}^{T} + \widetilde{\boldsymbol{r}} \widetilde{\boldsymbol{r}}^{T} + \widetilde{\boldsymbol{s}} \widetilde{\boldsymbol{s}}^{T} = \boldsymbol{I}$, $\boldsymbol{\mathcal{P}}\boldsymbol{u}=\boldsymbol{0}$ and $\boldsymbol{\mathcal{P}}\boldsymbol{\mathcal{P}}=\boldsymbol{\mathcal{P}}$, the collision kernel $\boldsymbol{\omega}(\boldsymbol{v}, \boldsymbol{v}')$ in Eq. \eqref{eq:CM2} can be rewritten as
\begin{equation}\label{eq:CM2_sep}
	\boldsymbol{\omega} (\boldsymbol{v}, \boldsymbol{v}') = g_{1}^{2} |\boldsymbol{\mathcal{P}}\boldsymbol{r}|^{2} \boldsymbol{\mathcal{P}} + (g_{2}^{2}-g_{1}^{2}) \boldsymbol{\mathcal{P}} \boldsymbol{r} \boldsymbol{r}^{T} \boldsymbol{\mathcal{P}} ,
\end{equation}
where the encoder functions $g_{\ast}$ are constructed as $g_{\ast}=g_{\ast}(|\boldsymbol{u}|,|\boldsymbol{v}|,|\boldsymbol{v}'|)$ for $\ast = 1, 2$. Given $\boldsymbol{v}$ and $\boldsymbol{v}'$, $g_s^2$ and $g_r^2$ in Eq. \eqref{eq:CM2} can be determined by $g_s^2 = g_1^2 \vert \mathcal{P}\boldsymbol{r}\vert^2$ and $g_r^2 = g_2^2 \vert \mathcal{P}\boldsymbol{r}\vert^2$.
For each encoder function $g_{\ast}(\cdot)$, we seek a low-rank tensor product representation with a set of \tcb{scalar-valued univariate basis functions (i.e., $\mathbb{R}\rightarrow \mathbb{R}$}) to separate the individual and relative velocity dependencies, i.e., 
\begin{equation}\label{eq:sepfun}
    \begin{aligned}
        g_{\ast}(|\boldsymbol{u}|,|\boldsymbol{v}|,|\boldsymbol{v}'|) &= \sum_{j'}^{J'} \mathcal{L}_{\ast}^{j'}(|\boldsymbol{u}|) \mathcal{M}_{\ast}^{j'}(|\boldsymbol{v}|) \mathcal{N}_{\ast}^{j'}(|\boldsymbol{v}'|) + \mathcal{L}_{\ast}^{j'}(|\boldsymbol{u}|) \mathcal{N}_{\ast}^{j'}(|\boldsymbol{v}|) \mathcal{M}_{\ast}^{j'}(|\boldsymbol{v}'|), \\
        &= \sum_{j}^{J} L_{\ast}^{j}(|\boldsymbol{u}|) M_{\ast}^{j}(|\boldsymbol{v}|) N_{\ast}^{j}(|\boldsymbol{v}'|),
    \end{aligned}
\end{equation}
where $\{\mathcal{L}_{\ast}^{j'}, \mathcal{M}_{\ast}^{j'}, \mathcal{N}_{\ast}^{j'}\}_{j'=1}^{J'}$ are univariate basis functions, which are essentially encoders represented by neural networks and will be learned from the MD data as shown in Sec. \ref{subsec:NNConstruct}. $J'$ is the total number of modes used for the low-rank representation. The right-hand-side of Eq. \eqref{eq:sepfun} \tcb{preserves the permutation symmetry for $\boldsymbol{v}$ and $\boldsymbol{v}'$, i.e., $\bm \omega(\bm{v},\bm{v}') = \bm \omega(\bm{v}',\bm{v})$}, due to Eq. \eqref{eq:kernel_conditions}. Specifically, we have  
$L_{\ast}^{2j'-1}=L_{\ast}^{2j'}=\mathcal{L}_{\ast}^{j'}$, $M_{\ast}^{2j'-1}=N_{\ast}^{2j'}=\mathcal{M}_{\ast}^{j'}$ and $M_{\ast}^{2j'}=N_{\ast}^{2j'-1}=\mathcal{N}_{\ast}^{j'}$, for $j'=1,\cdots,J'$, and $J=2J'$.

By the new kernel representation \eqref{eq:CM2_sep}, the collision operator \eqref{eq:collision_symplectic} can be expressed as the summation of three parts, i.e., 
\begin{equation}\label{eq:sim_CM2}
    \begin{aligned}
        \dfrac{\partial f}{\partial t} =& \nabla \cdot \int g_{1}^{2} |\boldsymbol{\mathcal{P}}\boldsymbol{r}|^{2} \boldsymbol{\mathcal{P}} f(\boldsymbol{v}) f(\boldsymbol{v}') \left[ \nabla \log f(\boldsymbol{v}) - \nabla'\log f(\boldsymbol{v}') \right] \mathrm{d}\boldsymbol{v}' \\
        &+ \nabla \cdot \int g_{2}^{2} \boldsymbol{\mathcal{P}} \boldsymbol{r} \boldsymbol{r}^{T} \boldsymbol{\mathcal{P}} f(\boldsymbol{v}) f(\boldsymbol{v}') \left[ \nabla \log f(\boldsymbol{v}) - \nabla'\log f(\boldsymbol{v}') \right] \mathrm{d}\boldsymbol{v}' \\
        &- \nabla \cdot \int g_{1}^{2} \boldsymbol{\mathcal{P}} \boldsymbol{r} \boldsymbol{r}^{T} \boldsymbol{\mathcal{P}} f(\boldsymbol{v}) f(\boldsymbol{v}') \left[ \nabla \log f(\boldsymbol{v}) - \nabla'\log f(\boldsymbol{v}') \right] \mathrm{d}\boldsymbol{v}' \\
        =&: I_{1} + I_{2} - I_{3},
    \end{aligned}
\end{equation}
where $|\boldsymbol{\mathcal{P}}\boldsymbol{r}|^{2}$ and $\boldsymbol{\mathcal{P}} \boldsymbol{r} \boldsymbol{r}^{T} \boldsymbol{\mathcal{P}}$ can be written as expressions of $\boldsymbol{u}$, $\boldsymbol{v}$, and $\boldsymbol{v}'$,
\begin{equation}
    \begin{aligned}
        |\boldsymbol{\mathcal{P}}\boldsymbol{r}|^{2} &= 2|\boldsymbol{v}|^{2} + 2|\boldsymbol{v}|'^{2} - |\boldsymbol{u}|^{2} - \frac{1}{|\boldsymbol{u}|^{2}} (|\boldsymbol{v}|^{4} - 2|\boldsymbol{v}|^{2} |\boldsymbol{v'}|^{2} + |\boldsymbol{v'}|^{4}) , \\
        \boldsymbol{\mathcal{P}} \boldsymbol{r} \boldsymbol{r}^{T} \boldsymbol{\mathcal{P}} &= \dfrac{1}{|\boldsymbol{u}|^4} \left[ |\boldsymbol{u}|^{2}\boldsymbol{I}-\boldsymbol{u}\boldsymbol{u}^{T} \right] \left[ 2\boldsymbol{v}-\boldsymbol{u} \right] \left[ \boldsymbol{u}+2\boldsymbol{v}' \right]^{T} \left[ |\boldsymbol{u}|^{2}\boldsymbol{I}-\boldsymbol{u}\boldsymbol{u}^{T} \right] \\
        &= \dfrac{1}{|\boldsymbol{u}|^4}
        \left[ 2\boldsymbol{v}|\boldsymbol{u}|^{4}\boldsymbol{u}^{T} + 4\boldsymbol{v}|\boldsymbol{u}|^{4}\boldsymbol{v}'^{T} - |\boldsymbol{u}|^{4}\boldsymbol{u}\boldsymbol{u}^{T} - 2|\boldsymbol{u}|^{4}\boldsymbol{u}\boldsymbol{v}'^{T} \right. \\
        &\qquad\quad -2|\boldsymbol{v}|^{2}\boldsymbol{v}|\boldsymbol{u}|^{2}\boldsymbol{u}^{T} - 2|\boldsymbol{v}|^{2}|\boldsymbol{u}|^{2}\boldsymbol{u}\boldsymbol{v}'^{T} + 2\boldsymbol{v}|\boldsymbol{u}|^{2}\boldsymbol{u}^{T}|\boldsymbol{v}'|^{2} + 2|\boldsymbol{u}|^{2}\boldsymbol{u}|\boldsymbol{v}'|^{2}\boldsymbol{v}'^{T} \\
        &\left. \qquad\quad +|\boldsymbol{v}|^{4}\boldsymbol{u}\boldsymbol{u}^{T} - 2|\boldsymbol{v}|^{2}\boldsymbol{u}\boldsymbol{u}^{T}|\boldsymbol{v}'|^{2} + \boldsymbol{u}\boldsymbol{u}^{T}|\boldsymbol{v}'|^{4} \right].
    \end{aligned}
\end{equation}

Before proceeding to the detailed form, let us consider an integral term in the form of $\int h_1\left (\vert \boldsymbol{v}\vert\right)h_2\left (\vert \boldsymbol{u}\vert\right)h_3\left (\vert \boldsymbol{v}'\vert\right) \boldsymbol{h}_4\left (\boldsymbol{v}\right) \boldsymbol{h}_5\left (\boldsymbol{u}\right) \boldsymbol{h}_6\left (\boldsymbol{v}'\right) \mathrm{d}\boldsymbol{v}'$, where $h_1$, $h_2$ and $h_3$ are scalar functions and $\boldsymbol{h}_4$, $\boldsymbol{h}_5$ and $\boldsymbol{h}_6$ are vector-valued or matrix-valued functions. It is easy to see that the integral term preserves a convolution structure, i.e.,
\begin{equation*}
\int h_1\left (\vert \boldsymbol{v}\vert\right)h_2\left (\vert \boldsymbol{u}\vert\right)h_3\left (\vert \boldsymbol{v}'\vert\right) \boldsymbol{h}_4\left (\boldsymbol{v}\right) \boldsymbol{h}_5\left (\boldsymbol{u}\right) \boldsymbol{h}_6\left (\boldsymbol{v}'\right) \mathrm{d}\boldsymbol{v}' =   h_1\left (\vert \boldsymbol{v}\vert\right) \boldsymbol{h}_4\left (\boldsymbol{v}\right) \int h_2\left (\vert \boldsymbol{u}\vert\right)\boldsymbol{h}_5\left (\boldsymbol{u}\right)  h_3\left (\vert \boldsymbol{v}'\vert\right) \boldsymbol{h}_6\left (\boldsymbol{v}'\right) \mathrm{d}\boldsymbol{v}'.    
\end{equation*}

By plugging the tensor representation \eqref{eq:sepfun} into Eq. \eqref{eq:sim_CM2}, $g_{\ast}^2$ induces terms as $h_1\left (\vert \boldsymbol{u}\vert\right)h_2\left (\vert \boldsymbol{v}\vert\right)h_3\left (\vert \boldsymbol{v}'\vert\right)$. Accordingly, the integral term $I_1$ in Eq. \eqref{eq:sim_CM2} can be written as 
\begin{equation}\label{eq:I1}
    \begin{aligned}
        I_{1} &= \nabla \cdot \int g_{1}^{2} |\boldsymbol{\mathcal{P}}\boldsymbol{r}|^{2} \boldsymbol{\mathcal{P}} f(\boldsymbol{v}) f(\boldsymbol{v}') \left[ \nabla \log f(\boldsymbol{v}) - \nabla'\log f(\boldsymbol{v}') \right] \mathrm{d}\boldsymbol{v}' \\
        &= \sum_{j,k} \nabla \cdot 
        \left[ \{2|\boldsymbol{v}|^2,\boldsymbol{\mathcal{P}},1\}_{1}^{j,k} 
        +\{2,\boldsymbol{\mathcal{P}},|\boldsymbol{v}'|^{2}\}_{1}^{j,k} 
        +\{-1,|\boldsymbol{u}|^{2}\boldsymbol{\mathcal{P}},1\}_{1}^{j,k} \right. \\
        &\qquad\qquad \left. 
        +\{-|\boldsymbol{v}|^{4},|\boldsymbol{u}|^{-2}\boldsymbol{\mathcal{P}},1\}_{1}^{j,k} 
        +\{ 2|\boldsymbol{v}|^{2},|\boldsymbol{u}|^{-2}\boldsymbol{\mathcal{P}},|\boldsymbol{v}'|^{2}\}_{1}^{j,k} 
        +\{-1,|\boldsymbol{u}|^{-2}\boldsymbol{\mathcal{P}},|\boldsymbol{v}'|^{4}\}_{1}^{j,k} \right].
    \end{aligned}
\end{equation}
where individual terms $\{ \boldsymbol{\alpha}(\boldsymbol{v}),\boldsymbol{\beta}(\boldsymbol{u}),\boldsymbol{\gamma}(\boldsymbol{v}')\}_{\ast}^{j,k}$ take the form
\begin{equation}
    \begin{aligned}
        \{ \boldsymbol{\alpha}(\boldsymbol{v}),\boldsymbol{\beta}(\boldsymbol{u}),\boldsymbol{\gamma}(\boldsymbol{v}')\}_{\ast}^{j,k} =& M_{\ast}^{j}(|\boldsymbol{v}|) M_{\ast}^{k}(|\boldsymbol{v}|) \boldsymbol{\alpha}(\boldsymbol{v}) \cdot \\
        & \left[ \langle\boldsymbol{\beta}(\boldsymbol{u}), \boldsymbol{\gamma}(\boldsymbol{v}')f(\boldsymbol{v}') \rangle_{\ast}^{j,k} f(\boldsymbol{v}) \nabla \log f(\boldsymbol{v}) - \langle \boldsymbol{\beta}(\boldsymbol{u}), \boldsymbol{\gamma}(\boldsymbol{v}') f(\boldsymbol{v}') \nabla' \log f(\boldsymbol{v}') \rangle_{\ast}^{j,k} f(\boldsymbol{v}) \right],
    \end{aligned}
\nonumber
\end{equation}
and $\langle \boldsymbol{\xi}(\boldsymbol{u}),\boldsymbol{\eta}(\boldsymbol{v}') \rangle_{\ast}^{j,k} $ represents the convolution structure, i.e.,  
\begin{equation*}
\langle \boldsymbol{\xi}(\boldsymbol{u}),\boldsymbol{\eta}(\boldsymbol{v}') \rangle_{\ast}^{j,k} = \int  L_{\ast}^{j}(|\boldsymbol{u}|) L_{\ast}^{k}(|\boldsymbol{u}|) \boldsymbol{\xi}(\boldsymbol{u}) \cdot N_{\ast}^{j}(|\boldsymbol{v}'|) N_{\ast}^{k}(|\boldsymbol{v}'|) \boldsymbol{\eta}(\boldsymbol{v}') \mathrm{d}\boldsymbol{v}',    
\end{equation*} 
which can be efficiently computed using the FFT method. The integral term $I_2$ in Eq. \eqref{eq:sim_CM2} can be expressed as
\begin{equation}\label{eq:I2}
    \begin{aligned}
        I_{2} &= \nabla \cdot \int g_{2}^{2} \boldsymbol{\mathcal{P}} \boldsymbol{r} \boldsymbol{r}^{T} \boldsymbol{\mathcal{P}} f(\boldsymbol{v}) f(\boldsymbol{v}') \left[ \nabla \log f(\boldsymbol{v}) - \nabla'\log f(\boldsymbol{v}') \right] \mathrm{d}\boldsymbol{v}' \\
        &= \sum_{j,k} \nabla \cdot 
        \left[ \{2\boldsymbol{v},\boldsymbol{u}^{T},1\}_{2}^{j,k} 
        + \{4\boldsymbol{v},1,\boldsymbol{v}'^{T}\}_{2}^{j,k} 
        + \{-1,\boldsymbol{u}\boldsymbol{u}^{T},1\}_{2}^{j,k} \right. \\
        &\qquad\qquad \left. + \{-2,\boldsymbol{u},\boldsymbol{v}'^{T}\}_{2}^{j,k} 
        + \{-2|\boldsymbol{v}|^{2}\boldsymbol{v},|\boldsymbol{u}|^{-2}\boldsymbol{u}^{T},1\}_{2}^{j,k} 
        + \{-2|\boldsymbol{v}|^{2},|\boldsymbol{u}|^{-2}\boldsymbol{u},\boldsymbol{v}'^{T}\}_{2}^{j,k} \right. \\
        &\qquad\qquad \left. 
        + \{2\boldsymbol{v},|\boldsymbol{u}|^{-2}\boldsymbol{u}^{T},|\boldsymbol{v}'|^{2}\}_{2}^{j,k} 
        + \{2,|\boldsymbol{u}|^{-2}\boldsymbol{u},|\boldsymbol{v}'|^{2}\boldsymbol{v}'^{T}\}_{2}^{j,k} 
        + \{|\boldsymbol{v}|^{4},|\boldsymbol{u}|^{-4}\boldsymbol{u}\boldsymbol{u}^{T},1\}_{2}^{j,k} \right. \\
        &\qquad\qquad \left. + \{-2|\boldsymbol{v}|^{2},|\boldsymbol{u}|^{-4}\boldsymbol{u}\boldsymbol{u}^{T},|\boldsymbol{v}'|^{2}\}_{2}^{j,k} 
        + \{1,|\boldsymbol{u}|^{-4}\boldsymbol{u}\boldsymbol{u}^{T},|\boldsymbol{v}'|^{4}\}_{2}^{j,k} \right] , 
    \end{aligned}
\end{equation}
which can also be efficiently computed by exploiting the convolution structure. The integral term $I_{3}$ can be represented similar to $I_{2}$ with $\ast=1$. Therefore, the tensor representation \eqref{eq:sepfun} of the new anisotropic non-stationary kernel \eqref{eq:CM2_sep} enables us to achieve efficient evaluation of the collision operator \eqref{eq:collision_symplectic}.  

\begin{remark}
If we use the original form of Eq. \eqref{eq:collision_symplectic}, 
the above fast approximation can still be used with the term $f(\boldsymbol{v}) \nabla \log f(\boldsymbol{v})$  replaced by $\nabla f(\boldsymbol{v})$ in Eq. \eqref{eq:sim_CM2}. We choose the log-form $f(\boldsymbol{v}) \nabla \log f(\boldsymbol{v})$ to construct the structure-preserving numerical scheme presented in Sec. \ref{subsec:discretization}.    
\end{remark}

\subsection{Data-driven construction from micro-scale MD}
\label{subsec:NNConstruct}

To construct the new collision kernel \eqref{eq:CM2_sep} \eqref{eq:sepfun}, we establish a data-driven learning of the univariate basis functions $\{L_{\ast}^{j}, M_{\ast}^{j}, N_{\ast}^{j}\}_{j=1}^{J}$ by matching the prediction of the kinetic processes between the kinetic equation \eqref{eq:collision_symplectic} and the micro-scale full MD data.
\tcr{Recent work \cite{chu2024inference} presents a data-driven inference of interaction kernels in opinion dynamics.}
Let $\left \{\boldsymbol{v}_{m}^{n}\right\}_{m=1}^{N_{MD}}$ denote a collection of the samples of the MD particle velocities at $t=t_n$. However, direct evaluation of Eq. \eqref{eq:collision_symplectic} involves terms such as $\nabla f(\boldsymbol{v})$ and $\nabla' f(\boldsymbol{v}')$ and therefore
relies on the numerical estimation of the three-dimensional probability density function of the MD velocity $f(\boldsymbol{v})$ from the empirical distribution $f^{\ast}(\boldsymbol{v}) = \frac{1}{N_{MD}}\sum_{m=1}^{N_{MD}} \delta(\boldsymbol{v}-\boldsymbol{v}_{m})$, which can be computationally expensive.
To alleviate this issue, we propose a weak form of the loss function by introducing a set of test basis functions $\left \{\psi_k(\boldsymbol{v})\right\}_{k=1}^K$, i.e.,
\begin{equation}\label{eq:loss}
    \mathcal{L}= \sum_{k=1}^K \sum_{n=1}^{N_T} \left(\left.\dfrac{\partial f^n}{\partial t}\right|_{MD}-\left.\dfrac{\partial f^n}{\partial t}\right|_{c},\psi_k(\boldsymbol{v})\right)^2,
\end{equation}
where $(h_1, h_2) := \int h_1(\boldsymbol{v}) h_2(\boldsymbol{v}) {\rm d} \boldsymbol{v}$ represents the inner product in the velocity space.  $K$ and $N_T$ represent the number of test basis functions and time steps, respectively. We refer to \ref{sec:MD} for the details on the MD and training setup. For the MD terms of the loss function, we have
\begin{equation}
    \left(\left.\dfrac{\partial f^n}{\partial t}\right|_{MD}, \psi_k(\boldsymbol{v})\right) \approx\left(\dfrac{f^{n+1}-f^{n}}{\Delta t},\psi_k(\boldsymbol{v})\right)
    =\dfrac{1}{N_{MD}\Delta t}\sum_{m=1}^{N_{MD}} (\psi_k(\boldsymbol{v}_{m}^{n+1})-\psi_k(\boldsymbol{v}_{m}^{n})),
\end{equation}
which can be efficiently evaluated by the Monte Carlo method with the empirical distribution $f^{\ast}(\boldsymbol{v})$.
\tcb{For the collision part, we use the weak form with gradient and tensor product}
\tcr{
\begin{equation}
    \begin{split}
        \left(\left.\dfrac{\partial f}{\partial t}\right|_{c},\psi_k(\boldsymbol{v})\right) 
        =& \left(\nabla_{\boldsymbol{v}} \cdot \left( \nabla_{\boldsymbol{v}} f(\boldsymbol{v})\cdot \int \boldsymbol{\omega} f(\boldsymbol{v}') \mathrm{d}\boldsymbol{v}' \right)
        - \nabla_{\boldsymbol{v}} \cdot \left( f(\boldsymbol{v}) \int \boldsymbol{\omega} \cdot \nabla_{\boldsymbol{v}'} f(\boldsymbol{v}') \mathrm{d}\boldsymbol{v}' \right), \psi_k(\boldsymbol{v}) \right) \\
        =&	\left(f(\boldsymbol{v})\int \boldsymbol{\omega} f(\boldsymbol{v}')\mathrm{d}\boldsymbol{v}', \nabla_{\boldsymbol{v}} \nabla_{\boldsymbol{v}} \psi_k(\boldsymbol{v})\right) 
        + \left( f(\boldsymbol{v}) \int (\nabla_{\boldsymbol{v}} \cdot \boldsymbol{\omega}) f(\boldsymbol{v}') \mathrm{d}\boldsymbol{v}', \nabla_{\boldsymbol{v}} \psi_k (\boldsymbol{v}) \right) \\
        &- \left( f(\boldsymbol{v}) \int (\nabla_{\boldsymbol{v}'} \cdot \boldsymbol{\omega}) f(\boldsymbol{v}') \mathrm{d}\boldsymbol{v}', \nabla_{\boldsymbol{v}} \psi_k (\boldsymbol{v}) \right), \\    
    \end{split}
\end{equation}
}
where the individual terms can also be evaluated by the Monte Carlo method with empirical distribution $f^{\ast}(\boldsymbol{v})$, i.e., 
\tcr{
\begin{equation}
    \begin{aligned}
        \left( f(\boldsymbol{v})\int \boldsymbol{\omega} f(\boldsymbol{v}')\mathrm{d}\boldsymbol{v}', \nabla_{\boldsymbol{v}} \nabla_{\boldsymbol{v}} \psi_k(\boldsymbol{v}) \right) =&
        \dfrac{1}{N_{MD}^2}\sum_{m,m'}^{N_{MD}} \boldsymbol{\omega}(\boldsymbol{v}_{m},\boldsymbol{v}_{m'} ') : \nabla_{\boldsymbol{v}} \nabla_{\boldsymbol{v}} \psi_k(\boldsymbol{v}_{m}), \\
        \left( f(\boldsymbol{v}) \int (\nabla_{\boldsymbol{v}} \cdot \boldsymbol{\omega}) f(\boldsymbol{v}') \mathrm{d}\boldsymbol{v}', \nabla_{\boldsymbol{v}} \psi_k (\boldsymbol{v}) \right) =&
        \dfrac{1}{N_{MD}^2}\sum_{m,m'}^{N_{MD}} (\nabla_{\boldsymbol{v}} \cdot \boldsymbol{\omega})(\boldsymbol{v}_{m},\boldsymbol{v}_{m'} ') \cdot \nabla_{\boldsymbol{v}} \psi_k (\boldsymbol{v}_{m}), \\
        \left( f(\boldsymbol{v}) \int (\nabla_{\boldsymbol{v}'} \cdot \boldsymbol{\omega}) f(\boldsymbol{v}') \mathrm{d}\boldsymbol{v}', \nabla_{\boldsymbol{v}} \psi_k (\boldsymbol{v}) \right) =&
        \dfrac{1}{N_{MD}^2}\sum_{m,m'}^{N_{MD}} (\nabla_{\boldsymbol{v}'} \cdot \boldsymbol{\omega})(\boldsymbol{v}_{m},\boldsymbol{v}_{m'} ') \cdot \nabla_{\boldsymbol{v}} \psi_k (\boldsymbol{v}_{m}) .
    \end{aligned}
\end{equation}
}

In practice, the number of MD particles $N_{MD}$ can be large (e.g. $N_{MD}\sim O(10^6)$). The double summation could be computationally expensive. Alternatively, we use the mini-batch method \cite{ketkar2017stochastic, li2014efficient, jin2020random, jin2021random, carrillo2021random} to accelerate the numerical evaluation, i.e., 
\tcr{
\begin{equation}\label{eq:minibatch}
    \begin{aligned}
        \left( f(\boldsymbol{v})\int \boldsymbol{\omega} f(\boldsymbol{v}')\mathrm{d}\boldsymbol{v}', \nabla_{\boldsymbol{v}} \nabla_{\boldsymbol{v}} \psi_k(\boldsymbol{v}) \right) &=
        \dfrac{1}{N_{MD}^2}\sum_{m,m'}^{N_{MD}} \boldsymbol{\omega}(\boldsymbol{v}_{m},\boldsymbol{v}_{m'} ') : \nabla_{\boldsymbol{v}} \nabla_{\boldsymbol{v}} \psi_k(\boldsymbol{v}_{m}) \\
        &\approx \dfrac{1}{P}\sum_{p=1}^{P} \boldsymbol{\omega}(\boldsymbol{v}_{m(p)},\boldsymbol{v}_{m'(p)} ') : \nabla_{\boldsymbol{v}} \nabla_{\boldsymbol{v}} \psi_k(\boldsymbol{v}_{m(p)}),
    \end{aligned}
\end{equation}
}
with pairs $(\boldsymbol{v}_{m(p)}, \boldsymbol{v}_{m'(p)} ')$ randomly selected from the MD samples for $p=1,\cdots,P$.

\subsection{Structure-preserving numerical discretization} \label{subsec:discretization}
With the new collision kernel $\boldsymbol\omega(\boldsymbol{v}, \boldsymbol{v}')$ \eqref{eq:CM2_sep} \eqref{eq:sepfun} constructed by minimizing the empirical loss function \eqref{eq:loss}, we obtain a generalized collision operator that retains the MD fidelity beyond empirical models, and meanwhile, strictly preserves the physical constraints. 
This enables us to further develop structure-preserving numerical schemes that inherit the essential physical properties on the discrete level.  To pursue the numerical solution, we consider the collisional kinetic equation in a bounded domain $\Omega$, i.e., 
\begin{equation}\label{eq:collision_symplectic_bounded}
    \begin{split}
        &\dfrac{\partial f}{\partial t}(\boldsymbol{v},t) = \nabla \cdot \int \boldsymbol{\omega}(\boldsymbol{v}, \boldsymbol{v}') \left[ f(\boldsymbol{v}') \nabla f(\boldsymbol{v}) - f(\boldsymbol{v}) \nabla' f(\boldsymbol{v}') \right] \mathrm{d}\boldsymbol{v}',  \quad \boldsymbol{v} \in \Omega \\
        &\left(  \int \boldsymbol{\omega}(\boldsymbol{v}, \boldsymbol{v}') \left[ f(\boldsymbol{v}') \nabla f(\boldsymbol{v}) - f(\boldsymbol{v}) \nabla' f(\boldsymbol{v}') \right] \mathrm{d}\boldsymbol{v}'\right) \cdot \hat{\boldsymbol{n}}(\boldsymbol{v}) = 0, ~~\quad \forall \boldsymbol{v} \in \partial\Omega.
    \end{split}
\end{equation}

It is straightforward to show that the conservation laws and H-theorem similar to Proposition \ref{prop:kernel_condition} hold for model \eqref{eq:collision_symplectic_bounded} in the bounded domain $\Omega$. 

In this work, we focus on the fast approximation of the generalized collision operator and follow the essential idea of Refs. \cite{degond1994entropy, buet1999numerical} to construct the semi-discrete structure-preserving scheme based on the finite-difference method.
\tcb{Other approaches such as the spectral method can also be straightforwardly integrated into the present spectral splitting strategy for numerical simulations of Eq. \eqref{eq:I1}.}
For simplicity, we consider a cubic domain and the distribution function is approximated by a piecewise constant function on an equal-space grid mesh at $\boldsymbol{v}_{\boldsymbol{i}}=(i_{1}h, i_{2}h, i_{3}h)$ with mesh size $h$ and $\boldsymbol{i}=(i_{1}, i_{2}, i_{3}) \in \mathbb{Z}^{3}$.
Denote $f_{\boldsymbol{i}}(t)$ as the approximated discrete probability at velocity $\boldsymbol{v}_{\boldsymbol{i}}$ and time $t$.
The evolution of the PDF is formulated as
\begin{equation}
    \begin{aligned}
        \dfrac{\partial f_{\boldsymbol{i}}(t)}{\partial t} =& \mathrm{D}^{+} \boldsymbol{p}_{\boldsymbol{i}} , \\ 
        \boldsymbol{p}_{\boldsymbol{i}} =& h^{3} \sum_{\boldsymbol{j} \in \mathbb{Z}^{3}} \boldsymbol{\omega}(\boldsymbol{v}_{\boldsymbol{i}}, \boldsymbol{v}_{\boldsymbol{j}}) f_{\boldsymbol{i}} f_{\boldsymbol{j}} \left[ \mathrm{D}^{-} \log f_{\boldsymbol{i}} - \mathrm{D}^{-} \log f_{\boldsymbol{j}} \right] ,
    \end{aligned}
\end{equation}
where $\boldsymbol{p}_{\boldsymbol{i}}$ is the probability flux at $\boldsymbol{v}_{\boldsymbol{i}}$, and $\mathrm{D}^{-}$ and $\mathrm{D}^{+}$ are dual discrete gradient and divergence operators of central difference discretization defined by
\begin{equation}
    \begin{aligned}
        (\mathrm{D}^{-} \psi_{\boldsymbol{i}})_{k} &= [\psi(\boldsymbol{v}_{\boldsymbol{i}}+\boldsymbol{e}_{k}h) - \psi(\boldsymbol{v}_{\boldsymbol{i}}-\boldsymbol{e}_{k}h)]/2h , \\
        \mathrm{D}^{+} \boldsymbol{\xi}_{\boldsymbol{i}} &= \sum_{k=1}^{3} [\xi_{k}(\boldsymbol{v}_{\boldsymbol{i}}+\boldsymbol{e}_{k}h) - \xi_{k}(\boldsymbol{v}_{\boldsymbol{i}}-\boldsymbol{e}_{k}h)]/2h ,
    \end{aligned}
\end{equation}
where $k=1,2,3$ and $\boldsymbol{e}_{k}$ is the unit vector in the $k$-th dimension.
In all numerical examples in this paper, we use the above central difference scheme for simulation.

\begin{proposition}
    The numerical semi-discretization scheme defined above satisfies the conservation law of mass, momentum, and kinetic energy, and the H-theorem in the discrete sense.
\end{proposition}

\begin{proof}
The definition of any discrete physical quantity on the grids and its evolution can be represented by
\begin{equation}
    \begin{aligned}
        \Phi(t) &= h^{3} \sum_{\boldsymbol{i}\in\mathbb{Z}^{3}} f_{\boldsymbol{i}} \phi_{\boldsymbol{i}} , \\
        \dfrac{\mathrm{d} \Phi(t)}{\mathrm{d} t} &=  h^{3} \sum_{\boldsymbol{i}\in\mathbb{Z}^{3}} \dfrac{\partial f_{\boldsymbol{i}}}{\partial t} \phi_{\boldsymbol{i}} \\
        &= h^{6} \sum_{\boldsymbol{i}\in\mathbb{Z}^{3}} \phi_{\boldsymbol{i}} \mathrm{D}^{+} \sum_{\boldsymbol{j}\in\mathbb{Z}^{3}} \boldsymbol{\omega}(\boldsymbol{v}_{\boldsymbol{i}}, \boldsymbol{v}_{\boldsymbol{j}}) f_{\boldsymbol{i}} f_{\boldsymbol{j}} \left[ \mathrm{D}^{-} \log f_{\boldsymbol{i}} - \mathrm{D}^{-} \log f_{\boldsymbol{j}} \right] \\
        &= -\dfrac{1}{2} h^{6} \sum_{\boldsymbol{i},\boldsymbol{j}\in\mathbb{Z}^{3}} \left[ \mathrm{D}^{-} \phi_{\boldsymbol{i}} - \mathrm{D}^{-} \phi_{\boldsymbol{j}} \right]^{T} \boldsymbol{\omega}(\boldsymbol{v}_{\boldsymbol{i}}, \boldsymbol{v}_{\boldsymbol{j}}) f_{\boldsymbol{i}} f_{\boldsymbol{j}} \left[ \mathrm{D}^{-} \log f_{\boldsymbol{i}} - \mathrm{D}^{-} \log f_{\boldsymbol{j}} \right] .
    \end{aligned}
\end{equation}
When we choose $\phi_{\boldsymbol{i}} = \{1, ~ v_{k},~ \boldsymbol{v}_{\boldsymbol{i}}^{2} \}$,
\begin{equation}
	\mathrm{D}^{-} \phi_{\boldsymbol{i}} - \mathrm{D}^{-} \phi_{\boldsymbol{j}} = \left\{
	\begin{aligned}
		&0 , & \phi_{\boldsymbol{i}} &= 1 \\
		&0 , & \phi_{\boldsymbol{i}} &= v_{k} \\
		&2(\boldsymbol{v}_{\boldsymbol{i}} - \boldsymbol{v}_{\boldsymbol{j}}) , & \phi_{\boldsymbol{i}} &= \boldsymbol{v}_{\boldsymbol{i}}^{2} .
	\end{aligned}
	\right.
\end{equation}
Moreover, $\boldsymbol{v}_{\boldsymbol{i}} - \boldsymbol{v}_{\boldsymbol{j}} \in \text{Ker}(\boldsymbol{\omega}(\boldsymbol{v}_{\boldsymbol{i}}, \boldsymbol{v}_{\boldsymbol{j}}))$, so that the conservation law holds.

The discrete entropy and its time derivative are defined as
\begin{equation}
    \begin{aligned}
        S(f) &= - h^{3} \sum_{\boldsymbol{i}\in\mathbb{Z}^{3}} f_{\boldsymbol{i}} \log f_{\boldsymbol{i}} , \\
        \dfrac{\mathrm{d} S(f)}{\mathrm{d} t} &= \dfrac{1}{2} h^{6} \sum_{\boldsymbol{i},\boldsymbol{j}\in\mathbb{Z}^{3}} \left[ \mathrm{D}^{-} \log f_{\boldsymbol{i}} - \mathrm{D}^{-} \log f_{\boldsymbol{j}} \right]^{T} \boldsymbol{\omega}(\boldsymbol{v}_{\boldsymbol{i}},\boldsymbol{v}_{\boldsymbol{j}}) f_{\boldsymbol{i}} f_{\boldsymbol{j}} \left[ \mathrm{D}^{-} \log f_{\boldsymbol{i}} - \mathrm{D}^{-} \log f_{\boldsymbol{j}} \right] \geq 0.
    \end{aligned}
\end{equation}
The equality holds if and only if $\mathrm{D}^{-} \log f_{\boldsymbol{i}} - \mathrm{D}^{-} \log f_{\boldsymbol{j}} \in \text{Ker}(\boldsymbol{\omega}(\boldsymbol{v}_{\boldsymbol{i}},\boldsymbol{v}_{\boldsymbol{j}}))$, i.e. the equilibrium state is the discrete Maxwellian distribution.
\end{proof}

\begin{remark}
The dual discrete gradient and divergence operators can be formed by spatial discretization schemes such as the forward-backward difference or finite volume method.    
\end{remark}


\begin{remark}
In this work, we use the explicit time discretization with time step size $dt \propto h^{2}$, and the full discrete scheme can only guarantee the preservation of the conservation laws. Although numerical results in Sec. \ref{sec:numerical} show the nonnegative entropy dissipation, implicit time schemes should be used for the theoretical guarantee of the H-theorem and numerical simulations with large time steps, which will be pursued in future work. Alternatively, the numerical scheme \cite{buet1999numerical} based on the geometric average of the distribution function might be considered.       
\end{remark}

\section{Numerical results}\label{sec:numerical}

In this study, we consider the one-component plasma system in the moderate coupling regime where the plasma kinetics can not be fully captured by the classical Landau equation; see \ref{sec:physical_regime} for detailed discussions.  
All numerical simulations are conducted in reduced units, with the physical parameters of the corresponding system provided in \ref{sec:MD}. 
To construct the generalized collision kernel $\boldsymbol{\omega}(\boldsymbol{v}, \boldsymbol{v}')$ in the spectral separation form, we utilize several MD simulation trajectories initiated from equilibrium particle configurations with a variety of initial velocity distributions. We refer to \ref{sec:training} for the detailed choice of the training data and the test bases $\{\psi_k(\boldsymbol{v})\}_{k=1}^K$.

\subsection{Generalized collision kernel}
Fig. \ref{fig:grgs} presents the encoder functions of the constructed generalized collision operator. For visualization purpose, we show the expectations of encoder functions with respect to the variable $s$, namely \tcb{$\mathbb{E}_{s}[g_{r}^2(u, r, s)]$ and $\mathbb{E}_{s}[g_{s}^2(u, r, s)]$} in Eq. \eqref{eq:CM2}, $\mathbb{E}_{s}[g_{1}^2(u, r, s)]$ and $\mathbb{E}_{s}[g_{2}^2(u, r, s)]$ in Eq. \eqref{eq:CM2_sep}, 
\tcb{where $\mathbb{E}_{s}$ is calculated by taking the expectation over $|\bm{s}|$ with the Monte Carlo sampling of $\boldsymbol{v}$ and $\boldsymbol{v}'$ independently from the Maxwellian distribution corresponding to the system temperature.}
In contrast to the Landau equation, the proposed generalized collision kernel $\boldsymbol{\omega} (\boldsymbol{v}, \boldsymbol{v}')$ is anisotropic and non-stationary. In particular, the energy transfer magnitude along the average velocity direction $\widetilde{\boldsymbol{r}} \propto \mathcal{P}(\boldsymbol{v} + \boldsymbol{v}' - 2\bar{\boldsymbol{v}})$  is smaller than that along the perpendicular direction $\widetilde{\boldsymbol{s}} \propto \widetilde{\boldsymbol{u}}\times \widetilde{\boldsymbol{r}}$, i.e., $g_{r}^{2} \leq g_{s}^{2}$ (or equivalently $g_2^2 \leq g_1^2)$. This heterogeneity is due to the collective interaction between the pair of colliding particles and the environment, which leads to an additional energy transfer on the plane with projection $\mathcal{P}_r = \boldsymbol{I} - \hat{\boldsymbol{r}}\hat{\boldsymbol{r}}^T$ and $\hat{\boldsymbol{r}} = \boldsymbol{r}/\vert \boldsymbol{r}\vert$. To preserve the conservation laws and H-theorem, the second energy transfer is restricted in the null space of $\boldsymbol{u}$, and therefore leads to a heterogeneous collision kernel, since $\widetilde{\boldsymbol{r}}^T \mathcal{P} \mathcal{P}_r \mathcal{P} \widetilde{\boldsymbol{r}} \le \widetilde{\boldsymbol{s}}^T \mathcal{P} \mathcal{P}_r \mathcal{P} \widetilde{\boldsymbol{s}}$ \tcr{(we refer to the Ref. \cite{zhao2025data} for further discussion)}. This complex many-body effect cannot be easily characterized by existing empirical collision models. In contrast, the present
enriched anisotropic and non-stationary structure \eqref{eq:CM2} and \eqref{eq:CM2_sep} enables the constructed kernel to naturally encode the unresolved particle interactions beyond the empirical forms.


\begin{figure}[H]
    \centering
    \includegraphics[width=0.7\textwidth]{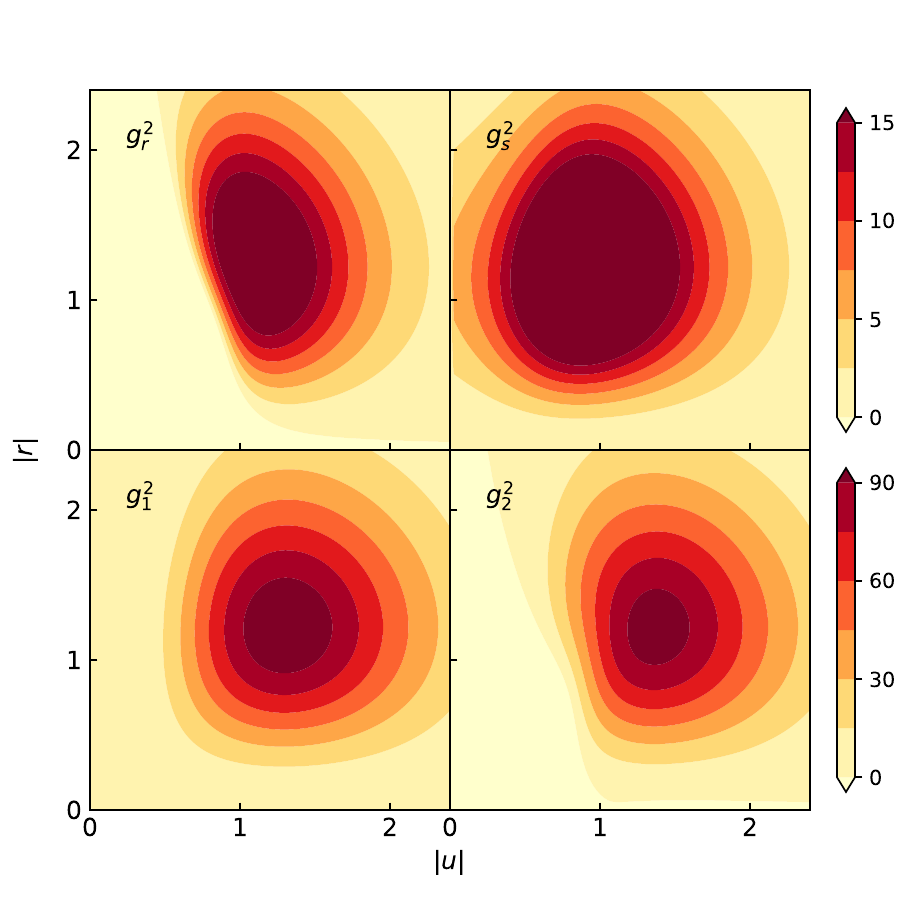}
    \caption{Expectations of encoder functions $\mathbb{E}_{s}[g_{r}^2(u, r, s)]$ and $\mathbb{E}_{s}[g_{r}^2(u, r, s)]$ in Eq. \eqref{eq:CM2}, and $\mathbb{E}_{s}[g_{1}^2(u, r, s)]$ and $\mathbb{E}_{s}[g_{2}^2(u, r, s)]$ in Eq. \eqref{eq:CM2_sep} over $s$. Unlike the standard Landau collision operator, the present model takes an anisotropic and non-stationary kernel. In particular, $g_r^2 \le g_s^2$ (or equivalently, $g_2^2 \le g_1^2$) due to the collective interactions between the colliding pair and the environment.}
    \label{fig:grgs}
\end{figure}

\subsection{Model verification}
To assess the accuracy of the proposed model, we perform comparative simulations of the plasma kinetics using three methods: MD simulations, the classical Landau equation, and the generalized collision operator with the spectral separation (SS) formulation.
In the numerical simulations of both the Landau equation and the proposed generalized collision operator, the computational domain in velocity space is defined as $[-L/2, L/2]^{3}$, with uniform mesh size $h$.
The number of grid points in each dimension is $N_{0} = L/h$, leading to a total of $N = N_{0}^{3}$ discrete velocity grid points.
A finite central difference scheme is applied to simulate the integro-differential equation given in Eq. \eqref{eq:collision_symplectic}.

Each simulation begins with one of two distinct initial velocity distributions. 
The first is a Gaussian mixture model (GMM), defined as follows:
\begin{equation}
    f(\boldsymbol{v},t=0) \sim \prod_{i=1}^{3} \left[ \alpha_{1} \exp\left(-\dfrac{|v_{i}-d_{i_{1}}|^{2}}{2\sigma_{1}^{2}}\right) + \alpha_{2} \exp\left(-\dfrac{|v_{i}-d_{i_{2}}|^{2}}{2\sigma_{2}^{2}}\right) \right],
    \label{eq:GMM_dist}
\end{equation}

The second initial velocity distribution is the radial model (RM), which combines a Gaussian distribution with a cosine modulation:
\begin{equation}
    f(\boldsymbol{v},t=0) \sim \alpha_{1} \exp(-\alpha_{2} \boldsymbol{v}^{2}/2) \cos(\alpha_{2} \boldsymbol{v}^{2})^{2}.
    \label{eq:RM_dist}
\end{equation}

The coefficients $\alpha_{\ast}$, $d_{\ast}$ and $\sigma_{\ast}$ are prescribed to ensure that the initial distributions satisfy three key conditions: unit total probability, zero net momentum, and a total kinetic energy consistent with the system parameters specified in the \ref{sec:simulation_settings}.
The initial velocity distribution in the GMM is constructed by extending a one-dimensional double-well potential to three dimensions via multiplicative separation along each velocity component.


Fig. \ref{fig:sim_GMM} presents the instantaneous velocity distribution function $f(\boldsymbol{v}, t)$ on the $v_x$-$v_y$ plane with the GMM initial distribution, obtained from the MD, Landau and the present model at $t=0.5$ and $t=1.0$, \tcb{and the comparison of the absolute error is shown in Fig. \ref{fig:GMM_abserr} in \ref{sec:simulation_settings}}.
The predictions of the present generalized collision kernel with SS structure can accurately reproduce the MD results, including the peaks of the distribution along the angular direction. 
In contrast, the prediction from the Landau model shows apparent deviations. 
Fig. \ref{fig:sim_radial} presents the simulation results of the RM case at times $t=0.25$, $0.5$, and $0.75$, comparing the velocity distribution functions $f(|\boldsymbol{v}|)$ and $|\boldsymbol{v}|f(|\boldsymbol{v}|)$ computed by using the MD, Landau, and the present model. Similar to the GMM case, the prediction of the present model shows good agreement with the MD results, while the prediction of the Landau model shows apparent deviations. In particular, the time evolution of the velocity distribution $f(\boldsymbol{v}, t)$ retains the radial symmetry condition discussed in Corollary \ref{cor:radial_symmetry}, and verifies the frame-indifference constraints.

\begin{figure}[H]
    \centering
    \includegraphics[width=0.7\textwidth]{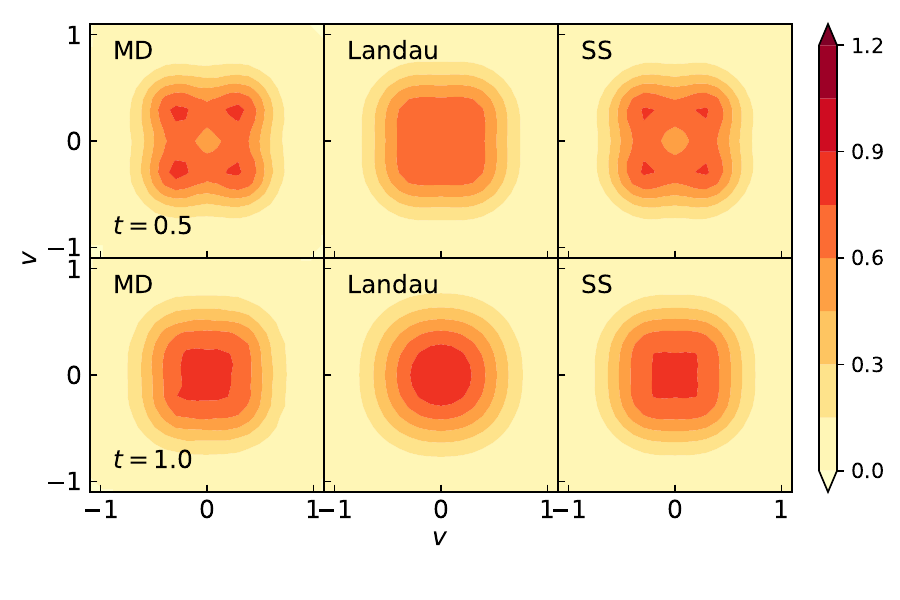}
    \caption{Predictions of the instantaneous velocity distribution on the $v_x$-$v_y$ plane at $t=0.5$ and $1.0$ with the initial distribution taking the GMM model defined by Eq. \eqref{eq:GMM_dist}. ``SS'' represents the present collision kernel with the spectral separation structure. }
    \label{fig:sim_GMM}
\end{figure}

\begin{figure}[H]
    \centering
    \includegraphics[width=0.7\textwidth]{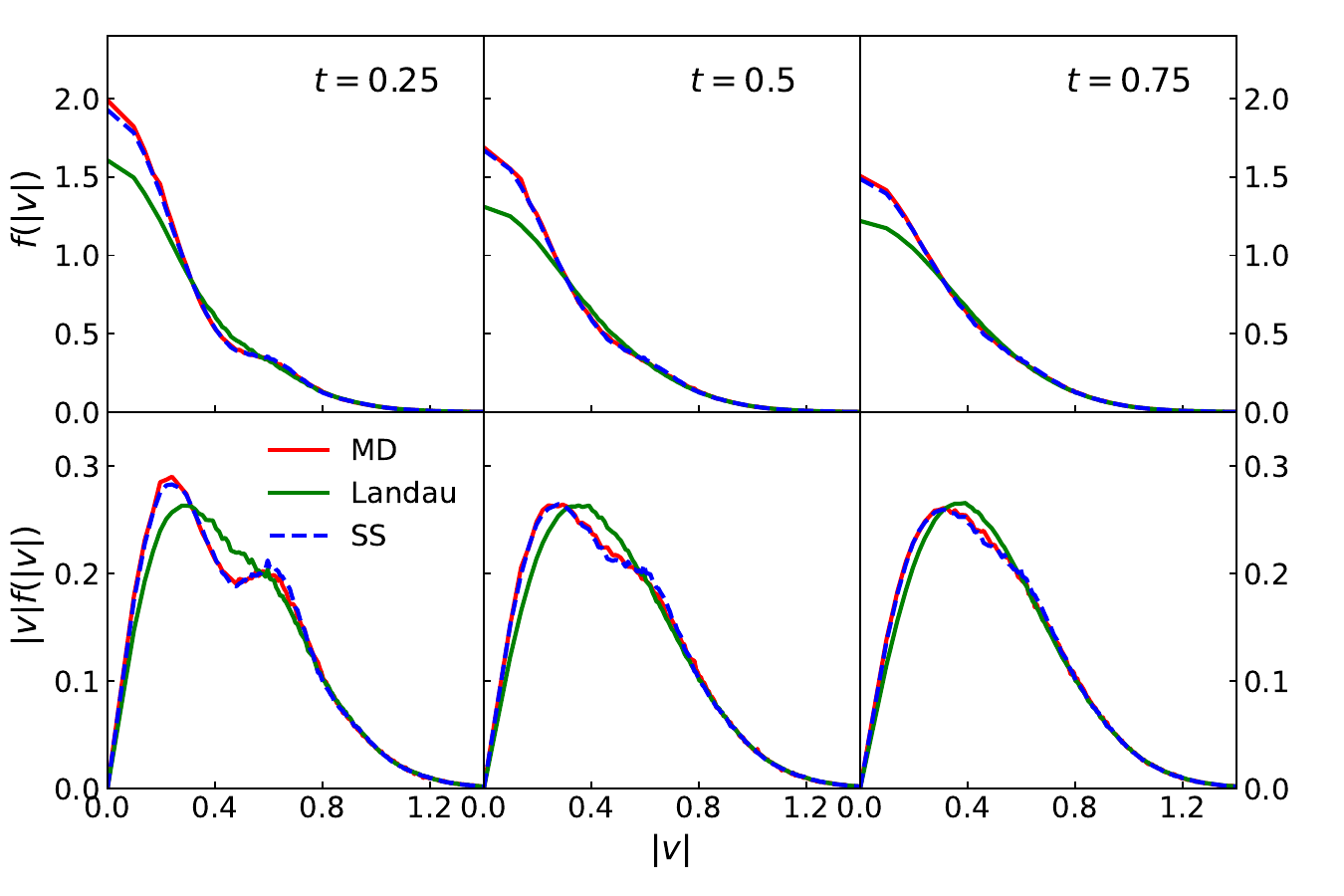}
    \caption{Predictions of the instantaneous velocity distribution along the radial direction with the initial condition taking the RM distribution defined by Eq. \eqref{eq:RM_dist}. }
    \label{fig:sim_radial}
\end{figure}

\subsection{Computational efficiency}
Next, we examine the computational efficiency of the present model. As discussed in Sec. \ref{sec:generalized_kernel}, direct simulation of the kinetic equation \eqref{eq:collision_symplectic},
\tcb{i.e. straightforward finite-difference discretization and evolution of the generalized non-stationary collision kernel $\boldsymbol{\omega}(\boldsymbol{v},\boldsymbol{v}')$ by Eq. \eqref{eq:CM2}, without using the FFT-based acceleration enabled by the spectral separation structure, requires a computational complexity of $\cO(N^{2})$ per time step and is prohibitively expensive for high-resolution simulations.} 
Instead, the present model with a low-rank tensor representation \eqref{eq:CM2_sep} enables the efficient evaluation of the collision integrals in Eqs. \eqref{eq:I1} and \eqref{eq:I2} via FFT, reducing the computational complexity to $\cO(N\log N)$. 
Figure \ref{fig:time} compares the simulation cost per timestep between the direct method and the fast spectral separation scheme, demonstrating significant computational savings.
\begin{figure}[H]
    \centering
    \includegraphics[width=0.45\textwidth]{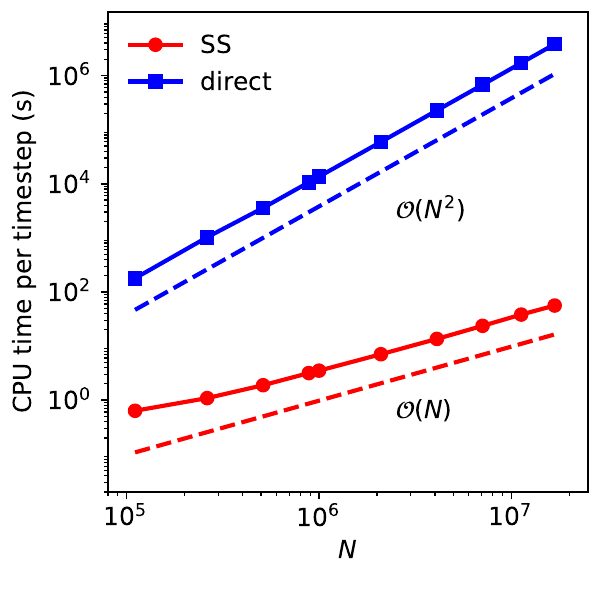}
    \caption{The simulation time per timestep of the direct simulation method and the present spectral separation (SS) method. The dotted lines represent $\cO(N)$ and $\cO(N^{2})$.}
    \label{fig:time}
\end{figure}

\subsection{Structure-preserving numerical results}

The present collision model with the SS representation strictly preserves the conservation laws and physical constraints, and therefore enables the further development of structure-preserving numerical schemes presented in Sec. \ref{subsec:discretization}. 
Fig. \ref{fig:err_MPE} presents the time evolution of the discrete errors in total mass $M$, momentum components $P_{i}$, and kinetic energy $K$, calculated from simulations initialized with the GMM and the RM distribution. For both cases, the numerical errors are consistently maintained at or below the machine precision throughout the simulations.
\begin{figure}[H]
    \centering
    \includegraphics[width=0.8\textwidth]{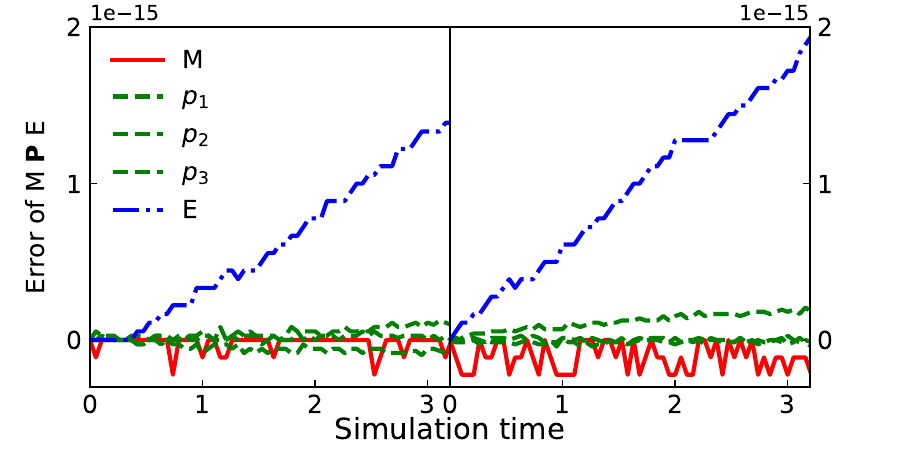}
    \caption{The error of total mass $M$, momentum $\boldsymbol{P}$, and kinetic energy $E$ along with the simulation time. The grid number is $N=128^{3}$, and the initial condition takes the GMM (left) and RM (right) distribution.}
    \label{fig:err_MPE}
\end{figure}

To numerically verify the H-theorem, we present the time evolution of the entropy for various grid sizes. As shown in Fig. \ref{fig:entropy}, the numerical results exhibit non-negative entropy production for both the GMM and the RM case. In particular, the GMM case takes a longer relaxation process since the initial distribution has a larger deviation from the Maxwellian distribution. On the other hand, we emphasize that the present numerical scheme adopts explicit time discretization and is insufficient for the theoretical guarantee of the H-theorem. Numerical methods based on implicit time discretization will be pursued in the future work. 

\begin{figure}[H]
    \centering
    \includegraphics[width=0.8\textwidth]{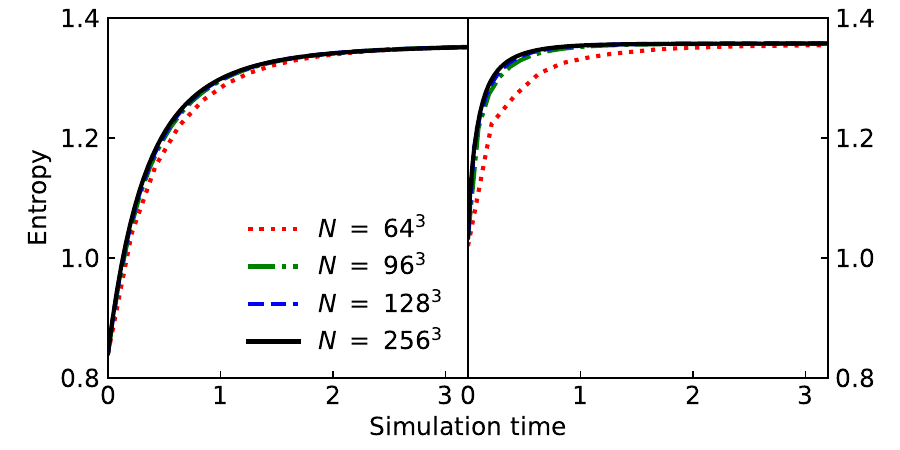}
    \caption{The time evolution of entropy with different grid sizes. The initial condition takes the GMM (left) and RM (right) distributions.}
    \label{fig:entropy}
\end{figure}

To examine the numerical convergence of the proposed SS method and discretization scheme, we quantify the relative errors in the $L_{1}$, $L_{2}$, and $L_{\infty}$ norms at different times.
The solution computed on the finest grid with $N = 256^{3}$ is taken as the reference solution, denoted by $f_{ref}(\boldsymbol{v}_{\boldsymbol{i}},t)$.
The relative error is defined as
\begin{equation}
    \begin{aligned}
        err_{1}(t) &= \dfrac{\sum_{\boldsymbol{v}_{\boldsymbol{i}}} |f_{h}(\boldsymbol{v}_{\boldsymbol{i}},t) - f_{ref}(\boldsymbol{v}_{\boldsymbol{i}},t)|}{\sum_{\boldsymbol{v}_{\boldsymbol{i}}} |f_{ref}(\boldsymbol{v}_{\boldsymbol{i}},t)|} , \\
        err_{2}(t) &= \left( \dfrac{\sum_{\boldsymbol{v}_{\boldsymbol{i}}} |f_{h}(\boldsymbol{v}_{\boldsymbol{i}},t) - f_{ref}(\boldsymbol{v}_{\boldsymbol{i}},t)|^{2}}{\sum_{\boldsymbol{v}_{\boldsymbol{i}}} |f_{ref}(\boldsymbol{v}_{\boldsymbol{i}},t)|^{2}} \right)^{1/2} , \\
        err_{\infty}(t) &= \dfrac{\max_{\boldsymbol{v}_{\boldsymbol{i}}} |f_{h}(\boldsymbol{v}_{\boldsymbol{i}},t) - f_{ref}(\boldsymbol{v}_{\boldsymbol{i}},t)|}{\max_{\boldsymbol{v}_{\boldsymbol{i}}} |f_{ref}(\boldsymbol{v}_{\boldsymbol{i}},t)|} .
    \end{aligned}
\end{equation}

Fig. \ref{fig:error_L2} presents the evolution of the $L_{2}$ relative errors in the velocity distribution for various mesh resolutions, initialized with the GMM and RM models, respectively. The numerical simulation error decreases with increasing grid resolution, and the error further diminishes for $t>1$ as the velocity distribution progressively converges to the Maxwellian equilibrium. 

\begin{figure}[H]
    \centering
    \includegraphics[width=0.8\textwidth]{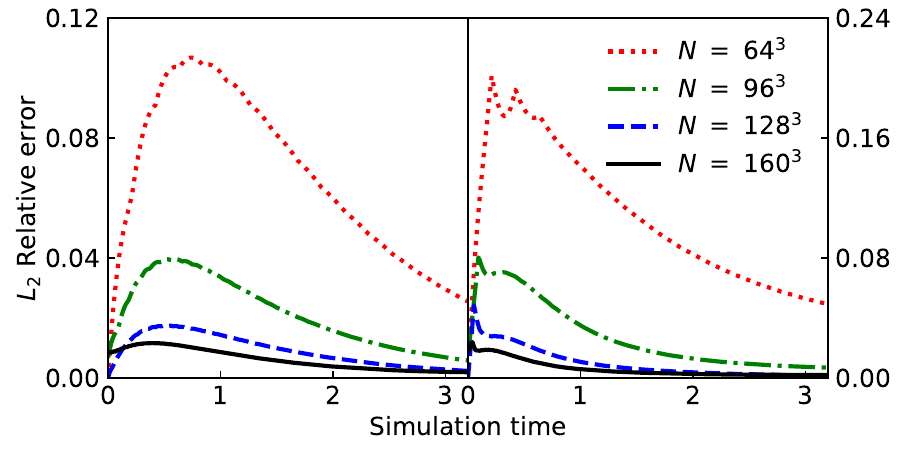}
    \caption{The $L_{2}$ simulation error with different mesh sizes, compared with the finest mesh $N=256^{3}$, beginning with two initial distributions, GMM (left) and RM (right).}
    \label{fig:error_L2}
\end{figure}

To assess the convergence quantitatively, we compare the $L_{2}$ relative error as a function of mesh size $h$ for both initial distribution models at $t=1$ and $t=2$, as shown in Fig. \ref{fig:convergence}. The numerical results indicate that the $L_{2}$ error exhibits a convergence rate between $\cO(h^{2})$ and $\cO(h^{3})$, with the $L_{1}$ and $L_{\infty}$ errors also falling within this range. \tcr{In practice, the convergence rate is slightly better than the expected $O(h^2)$. This is perhaps due to 
the relaxation nature, where the evolution eventually converges to the Maxwellian distribution}.
The corresponding $L_{1}$ and $L_{\infty}$ errors at different mesh resolutions are summarized in Table \ref{table:errorL1} and \ref{table:errorLinf}.

\begin{figure}[H]
    \centering
    \includegraphics[width=0.8\textwidth]{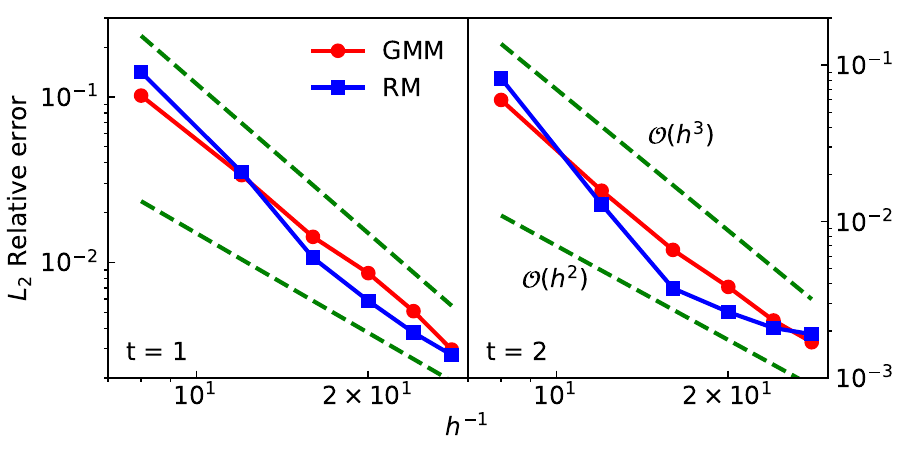}
    \caption{The $L_{2}$ relative error of two initial models at $t=1$ and $t=2$ with different mesh sizes.}
    \label{fig:convergence}
\end{figure}

\begin{table}[H]
    \centering
    \begin{tabular}{c c c c c}
         $h$ & GMM($t=1$) & GMM($t=2$) & RM($t=1$) & RM($t=2$) \\\hline
         $1/8$ & $7.95\times 10^{-2}$ & $4.35\times 10^{-2}$ & $1.37\times 10^{-1}$ & $8.47\times 10^{-2}$ \\
         $1/12$ & $2.57\times 10^{-2}$ & $1.25\times 10^{-2}$ & $3.69\times 10^{-2}$ & $1.47\times 10^{-2}$ \\
         $1/16$ & $1.08\times 10^{-2}$ & $5.27\times 10^{-3}$ & $1.13\times 10^{-2}$ & $3.82\times 10^{-3}$ \\
         $1/20$ & $7.00\times 10^{-3}$ & $3.65\times 10^{-3}$ & $5.48\times 10^{-3}$ & $2.03\times 10^{-3}$ \\
         $1/24$ & $4.38\times 10^{-3}$ & $2.44\times 10^{-3}$ & $3.22\times 10^{-3}$ & $1.74\times 10^{-3}$ \\
         $1/28$ & $2.99\times 10^{-3}$ & $1.97\times 10^{-3}$ & $2.21\times 10^{-3}$ & $1.64\times 10^{-3}$
    \end{tabular}
    \caption{$L_{1}$ relative error at different mesh resolutions, beginning with GMM and RM models at $t=1$ and $t=2$.}
    \label{table:errorL1}
\end{table}

\begin{table}[H]
    \centering
    \begin{tabular}{c c c c c}
         $h$ & GMM($t=1$) & GMM($t=2$) & RM($t=1$) & RM($t=2$) \\\hline
         $1/8$ & $2.76\times 10^{-1}$ & $1.94\times 10^{-1}$ & $1.24\times 10^{-1}$ & $8.81\times 10^{-2}$ \\
         $1/12$ & $1.19\times 10^{-1}$ & $5.79\times 10^{-2}$ & $3.26\times 10^{-2}$ & $2.17\times 10^{-2}$ \\
         $1/16$ & $5.22\times 10^{-2}$ & $2.25\times 10^{-2}$ & $9.64\times 10^{-3}$ & $1.05\times 10^{-2}$ \\
         $1/20$ & $2.49\times 10^{-2}$ & $1.06\times 10^{-2}$ & $8.53\times 10^{-3}$ & $6.10\times 10^{-3}$ \\
         $1/24$ & $1.25\times 10^{-2}$ & $5.06\times 10^{-3}$ & $7.71\times 10^{-3}$ & $5.86\times 10^{-3}$ \\
         $1/28$ & $5.39\times 10^{-3}$ & $2.10\times 10^{-3}$ & $4.79\times 10^{-3}$ & $3.49\times 10^{-3}$ \\
    \end{tabular}
    \caption{$L_{\infty}$ relative error at different mesh resolutions, beginning with GMM and RM models at $t=1$ and $t=2$.}
    \label{table:errorLinf}
\end{table}

\section{Summary}\label{sec:summary}

This study addresses an essential limitation of classical collisional kinetic model by introducing a generalized collision operator capable of accurately modeling plasma dynamics beyond the weakly coupled regime.
The proposed generalized collision operator incorporates an anisotropic and non-stationary kernel learned directly from MD simulations. This enables the new operator to capture the broadly overlooked anisotropic energy transfer arising from the collective interactions between the pair of colliding particles and the environment, where the traditional Landau formulation shows limitations.

To overcome the computational challenges due to the non-stationary form of the generalized operator, a fast spectral separation method is developed by seeking a low-rank tensor representation of the generalized collision kernel. This enables the efficient evaluation of the collision operator by exploiting the convolution structure of individual integral terms, and thereby reduces the computational complexity from $\cO(N^{2})$ to $\cO(N \log N)$.
The constructed operator strictly preserves the conservation laws and H-theorem, and enables the further development of structure-preserving numerical schemes. 

Numerical results demonstrate that the present model improves the Landau operator in the accuracy of the model while maintaining high computational efficiency. Various physical constraints can be numerically preserved at the discrete level. 
This work achieves an advance toward accurate and efficient modeling of collisional plasmas over a broader range of physical conditions. Future work will be devoted to developing an implicit numerical scheme with the theoretical guarantee of the H-theorem and constructing the full kinetic model with a generalized collisional operator that further accounts for spatial inhomogeneity.


\appendix
\section{Physical regime of plasma kinetics}\label{sec:physical_regime}

In this work, we investigate the kinetic processes of one-component plasma systems.
Following the standard collisional kinetic theory \cite{hinton1983collisional}, we use the dimensionless coupling parameter $\Gamma$ to characterize the degree of coupling between the potential and kinetic energy, i.e.,  
\begin{equation}
    \Gamma = \dfrac{q_{e}^{2}}{4\pi \epsilon_{0} k_{B} T}(4\pi n/3)^{1/3} ,
\end{equation}
where $q_{e}$ denotes the particle charge, $\epsilon_{0}$ is the vacuum permittivity, $k_{B}$ is the Boltzmann constant, $T$ is the temperature, and $n$ is the particle number density.

This parameter serves as a metric for the physical regime of the plasma kinetics. For $\Gamma \ll O(1)$, the plasma is in the weakly coupled regime (e.g., high-temperature, low-density), and the small-angle scattering and binary collisions are the dominant interaction mechanisms. The Landau equation provides a valid approximation that remains valid.
For $\Gamma \sim O(1)$, the plasma is in the moderately coupled regime, and the effects of the particle correlation become non-negligible. The Landau equation is insufficient to characterize the plasma kinetics due to the oversimplification of the collisional interactions in the form of a homogeneous kernel. In this study, we focus on this challenging regime and aim to accurately model the plasma kinetics that accounts for the unresolved particle correlations in the form of a generalized heterogeneous non-stationary collisional kernel \eqref{eq:CM2} \eqref{eq:CM2_sep}. 

\section{MD settings}
\label{sec:MD}
In this study, MD simulations are performed by using $N_{MD} = 10^{6}$ particles confined in a cubic box of side length $100~\AA$.
The particles interact via an unscreened Coulomb interaction, with each particle carrying a monovalent charge.
Under this condition, the coupling parameter of the system is $\Gamma=2.3$, corresponding to the moderately coupled plasma regime.
The MD simulations are initialized with various velocity distributions and equilibrium configurations.

In order to simplify the processing of neural network training and numerical simulations, we introduce characteristic scales for velocity and time
\begin{equation}
    v_{0} = a \omega_p \approx 81700 ~ \text{m/s}, \qquad t_{0} = \omega_{p}^{-1}\approx 0.76 ~ \text{fs} ,
\end{equation}
where $a = \left(3/4\pi n\right)^{1/3}$ is the Wigner-Seitz radius and $\omega_p = (n e^2/m \epsilon_0 )^{1/2}$ is the plasma frequency.
Velocity and time are nondimensionalized via $\boldsymbol{v} = \boldsymbol{v}_{MD}/v_{0}$, $t=t_{MD}/t_{0}$, and the variance of the velocity in $3$-D is $0.433$.

\section{Training settings}
\label{sec:training}

Only $3$ MD trajectories (in reduced units), with initial particle velocity following the uniform and the bi-Maxwellian distributions, are used as the training data with number of particles $N=10^{6}$, to learn the collision kernel in Eq. \eqref{eq:CM2_sep}. We emphasize that the GMM and RM models (see Eq. \eqref{eq:GMM_RM_dist}) are excluded from the training set and will be only used for the validation presented in Sec. \ref{sec:numerical}.  

\tcr{For the test functions $\psi_k(\boldsymbol{v})$ in Eq. \eqref{eq:loss}, we select $\exp[-(\boldsymbol{v}-\mu)^2/2\sigma^{2}]$, $\alpha~\boldsymbol{v}^2 \exp(-\boldsymbol{v}^2/2\sigma^{2})$, $\exp[-(\boldsymbol{v}^2 - \mu^{2})^2/2\sigma^{2}]$ with parameters $\mu$, $\sigma$, designed to cover the relevant regions of velocity space distribution based on the system temperature.
And we choose $k=5$ with $\psi_k(\boldsymbol{v})$ in Eq. \eqref{eq:loss}, with $\psi_{k}(\boldsymbol{v}) = \exp(-6~\boldsymbol{v}^2)$, $6~\boldsymbol{v}^2 \exp(-1.2\boldsymbol{v}^2)$, $\exp[-(6~\boldsymbol{v}^2 - 1.0)^2]$, $\exp[-(2.5~\boldsymbol{v} - 0.2)^2]$, and $\exp[-(2.5~\boldsymbol{v} - 0.5)^2]$, where the learned collision kernel and the resulting macroscopic observables are not sensitive to the test functions.} 
\tcb{Furthermore, we have tested different numbers and families of test functions and observed that the learned kernel and the predictions of the kinetic dynamics remains consistent.}

\tcr{The encoder functions $\{\mathcal{L}_{\ast}^{j'}, \mathcal{M}_{\ast}^{j'}, \mathcal{N}_{\ast}^{j'}\}_{j'=1}^{J'}$ are constructed as fully connected deep neural networks, each with: $5$ hidden layers, $10$ neurons per layer, and the activation function is set as the sigmoid function.
We choose $J'=4$, and use $3$ independent MD trajectories and $5$ test functions with batch size $P = 10^{6}$ in the training process.}

\section{Simulation settings}
\label{sec:simulation_settings}
For numerical simulations by using the finite central difference method, the computational domain is set as $[-L/2, L/2]^{3}$ with $L=8$.
A range of spatial mesh sizes is considered: $h=1/6$, $1/8$, $1/10$, $1/12$, $1/16$, $1/20$, $1/24$, $1/28$, and $1/32$, corresponding to the grid resolutions of $N=48^{3}$, $64^{3}$, $80^{3}$, $96^{3}$, $128^{3}$, $160^{3}$, $192^{3}$, $224^{3}$, and $256^{3}$, respectively.
The simulation time step is chosen to scale with the mesh size as $dt \propto h^{2}$ to ensure numerical stability and accuracy.

The initial GMM and RM distributions take the form 
\begin{equation}
    \begin{aligned}
        f_{GMM}(\boldsymbol{v},0) &= 1.65 \prod_{i=1}^{3} \left[ \exp\left(-\dfrac{|v_{i}-0.34|^{2}}{0.057}\right) + \exp\left(-\dfrac{|v_{i}+0.34|^{2}}{0.057}\right) \right] , \\
        f_{RM}(\boldsymbol{v},0) &= 1.254 \exp(-3.535~ \vert \boldsymbol{v}\vert ^{2}) \cos(7.07~\vert \boldsymbol{v}\vert ^{2})^{2} .
    \end{aligned}
    \label{eq:GMM_RM_dist}
\end{equation}

\tcb{The comparison of the absolute error of the Landau and the present model of the GMM case is shown in Fig. \ref{fig:GMM_abserr}; see also Fig. \ref{fig:sim_GMM} for the time evolution of the $f(\bm v, t)$ predicted by the full MD, the Landau, and the present model. These results show that the present model achieves significantly smaller error.}
\begin{figure}[H]
    \centering
    \includegraphics[width=0.45\textwidth]{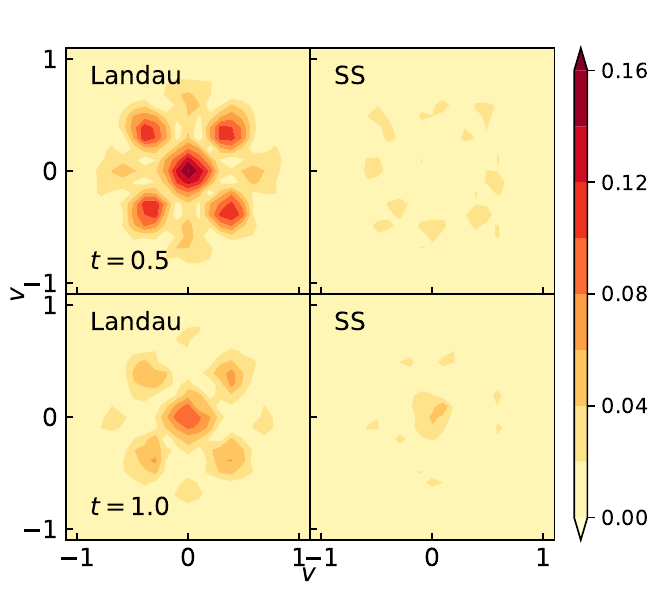}
    \caption{Absolute error $\vert f - f_{\rm MD}$ of simulation results between Landau vs MD (left), and SS vs MD (right) on the $v_{x}-v_{y}$ plane at $t=0.5$ and $1.0$, with the initial distribution taking the GMM model.}
    \label{fig:GMM_abserr}
\end{figure}

\section*{Acknowledgments}
We acknowledge Prof. Guosheng Fu for insightful discussions. The work is supported in part by the National Science Foundation under Grant DMS-2110981 and DMS-2143739, and the Department of Energy under Grant No. DOE-DESC0023164  and the ACCESS program through allocation MTH210005. 
The authors also acknowledge the support from the Institute for Cyber-Enabled Research at Michigan State University.






\end{document}